\documentclass[11pt,a4paper]{amsart}

\usepackage{hyperref}
\usepackage{fullpage}
\usepackage{amscd,amsmath,amsthm,amssymb}
\usepackage{comment}
\usepackage{color}

\usepackage[labelformat=simple]{subcaption}

\usepackage{tikz, float} \usetikzlibrary {positioning}

\usepackage{url}

\def\mm{\mathfrak{m}}

\def\Phi{\mathfrak{n}}
\def\Phi{\mathfrak{N}}
\def\MM{{\mathcal M}}

\def\Gc{{\mathcal G}}

\def\opn#1#2{\def#1{\operatorname{#2}}} 
\opn\rank{rank}
\DeclareMathOperator{\spn}{span}
\opn\height{height}
\opn\ini{in}
\opn\lex{lex}
\opn\supp{supp}
\opn\Tor{Tor}
\def\rk{{\rm rk}}
\opn\V{\mathcal{V}}
\let\union=\cup
\let\iso=\cong

\opn\adj{adj}  

\newtheorem{Theorem}{Theorem}[section]
\newtheorem{Lemma}[Theorem]{Lemma}
\newtheorem{Corollary}[Theorem]{Corollary}
\newtheorem{Proposition}[Theorem]{Proposition}

\theoremstyle{remark}
\newtheorem{Remark}[Theorem]{Remark}

\newtheorem{Example}[Theorem]{Example}

\theoremstyle{definition}
\newtheorem{Definition}[Theorem]{Definition}

\let\epsilon\varepsilon
\let\kappa=\varkappa

\numberwithin{equation}{section}

\tikzstyle{Cgray}=[draw=black, scale = .75,circle, fill = white, minimum size=5mm, inner sep = 2pt] 
\tikzstyle{Cgray2}=[draw=black, scale = .60,circle, fill = white, minimum size=5mm, inner sep = 2pt] 
\tikzstyle{Cwhite}=[scale = .8,circle, fill = white, minimum size=8mm] 
\tikzstyle{Cblack}=[scale = .7,circle, fill = black, minimum size=3mm]
\tikzstyle{C0}=[scale = .9,circle, fill = black!0, inner sep = 0pt, minimum size=3mm]
\tikzstyle{C1}=[scale = .7,circle, fill = black!0, inner sep = 0pt, minimum size=3mm]
\tikzstyle{clique}=[fill=green!20!white]
\tikzstyle{clique1}=[fill=green!30!white,semitransparent]
\tikzstyle{clique2}=[fill=green!60!white,semitransparent]
\tikzstyle{clique3}=[fill=blue!20!white,semitransparent]
\tikzstyle{clique4}=[fill=blue!40!white,semitransparent]
\tikzstyle{clique5}=[fill=red!40!white,semitransparent]

\newcommand{\red}[1]{{\color{red} #1}}

\begin{document}

\title {Prime splittings of determinantal ideals}

\author {Fatemeh Mohammadi$^{1}$ and Johannes Rauh$^{2}$}

\address{$^{1}$School of Mathematics, University of Bristol, Bristol, BS8 1TW, United Kingdom}
\email{fatemeh.mohammadi@bristol.ac.uk}
\address{$^{2}$
MPI for Mathematics in the Sciences, Inselstra\textup{\ss}e 21, 04103 Leipzig, Germany}
\email{jrauh@mis.mpg.de}

\begin{abstract}
We consider determinantal ideals, where the generating minors are encoded in a hypergraph.
We study when the generating minors form a Gr\"obner basis.  In this case, the ideal is radical, and we can
describe algebraic and numerical invariants of these ideals in terms of combinatorial data of their hypergraphs, such as
the clique decomposition.  In particular, we can construct a minimal free resolution as a tensor
product of the minimal free resolution of their cliques.
For several classes of hypergraphs we find a combinatorial description of the minimal primes
in terms of a prime splitting.  That is, we write the determinantal ideal as a sum of smaller determinantal ideals such
that each minimal prime is a sum of minimal primes of the summands.
\end{abstract}

\maketitle

\tableofcontents

\section{Introduction}

Let $K$ be a field.
For fixed integers $m,n$ with $2<m\leq n$, let $X=(x_{ij})$ be a generic $m\times n$-matrix, and let $R=K[X]$ be the polynomial ring over $K$ in
indeterminates $x_{ij}$.  Our objects of study are determinantal ideals (that is, ideals $J$ generated
by sets of minors of~$X$), and our goal is to decompose $J$ into a sum of ideals, $J=J_{1}+J_{2}+\cdots+J_{r}$, such that
each $J_{i}$ is itself a determinantal ideal and such that the components $J_{i}$ are easier to understand algebraically and combinatorially.

Let $\Delta$ be a \emph{hypergraph} with vertex set~$[n]$; that is, $\Delta$ is a family of subsets of~$[n]$.  The
elements of $\Delta$ are the \emph{hyperedges} of $\Delta$.  The \emph{dimension} of a hyperedge
$T=\{b_{1},\dots,b_{k}\}\in\Delta$ is $\dim(T)=|T|-1$, and the dimension of $\Delta$ is
$\dim(\Delta)=\max_{T\in\Delta}\dim(T)$.
Given $\Delta$ and a subset $S\subseteq[m]$, we define the determinantal ideal
\begin{equation*}
  J_{\Delta}^{S}=([a_1\ldots a_k| b_1\ldots b_k]\;:\; \{b_1,\ldots,b_k\}\in \Delta \text{ and } \{a_1,\ldots,a_k\}\subset S\}),
\end{equation*}
where $[a_1\ldots a_k| b_1\ldots b_k]$ is the $k$-minor of the submatrix of $X$ with row indices $a_1,\ldots,a_k$ and
column indices $b_1,\ldots,b_k$.  We mostly focus on the case where~$S=[m]$, in which case we omit the superscript~$S$.
We call $J_{\Delta}^{S}$ a \emph{determinantal hypergraph ideal}\footnote{Other
  authors (e.g.~\cite{EHHM}) have preferred to work with the set of facets of a simplicial complex instead of a
  hypergraph.  However, observe that the set of determinants contained in a given ideal usually is not a simplicial
  complex.  In most of the ideals that we want to study, more or less the opposite is true: The set of determinants that
  are \emph{not} contained in the ideal define a simplicial complex.  This should be compared to the definition of the
  monomial Stanley-Reisner ideal associated with a simplicial complex, where the simplicial complex is also constructed
  from the set of monomials \emph{not} contained in the ideal.}.
Clearly, any determinantal ideal $J$ can be written as a sum of determinantal hypergraph ideals, and usually this
decomposition is not unique.  We want to find decompositions that allow to infer algebraic properties of $J$ from
algebraic properties of the summands.

In general, if an ideal $I$ is written as a sum of ideals $I_{1},\ldots,I_{r}$, it is not possible to directly extract
the algebraic invariants of $I$ (e.g. Hilbert function, primary components, minimal free resolution, etc.) from the
algebraic invariants of its subideals $I_i$.  Indeed, any ideal can be written as a sum of principal ideals.
However, in Sections~\ref{sec:block-adjacent-hyps} and~\ref{sec:further-examples}, we find ideals $I$ for which a nice
decomposition into smaller ideals $I_1,\ldots,I_r$ exists such that the minimal primes of $I$ can be determined from the
minimal primes of the ideals $I_i$ in the following sense:
\begin{Definition}
Let $I,I_1,\ldots,I_r$ be ideals such that $I=I_1+\cdots+I_r$.  Then $I_1+\cdots+I_r$ is a \emph{prime splitting} of $I$ if the following condition holds:
\begin{itemize}
\item[(*)] $P$ is a minimal prime ideal of $I$ if and only if there exist minimal prime ideals $P_{i}$ of $I_i$ such that $P=P_{1}+\cdots+P_{r}$.
\end{itemize}
\end{Definition}
Trivial examples of prime splittings occur if $K$ is algebraically closed and if the ideals $I_{1},\dots,I_{r}$ are defined in disjoint sets of variables.
Later, we will find more interesting examples of prime splittings that have the following weaker property: There exists
a term order $<$ such that the initial ideals $\ini_{<}I_{1},\dots,\ini_{<}I_{r}$ are defined in disjoint sets of variables.
This condition is not sufficient for a prime splitting, though, as the following example shows.
\begin{Example}
  \label{ex:disj-in-id-not-suff}
  The ideals $I_{1} = ( x^{2}-z )$ and $I_{2} = (y^{2} - z)$ are both prime, and their initial ideals
  with respect to the lexicographic order are generated in disjoint sets of variables.  However, the ideal $I =
  I_{1}+I_{2}$ has a non-trivial primary decomposition $I = (x + y, y^{2}-z) \cap (x - y, y^{2}-z)$.\qed
\end{Example}



To put our work in perspective, let us mention some prior results on determinantal ideals.
The classical ideal $I_{(k)}$ generated by all $k$-minors of $X$ is equal to $J_{\Delta_{(k)}}$, where $\Delta_{(k)}$ consists of all subsets of $[n]$ of cardinality~$k$. 
The ideal $I_{(k)}$ has been extensively studied from the viewpoint of algebraic geometry and commutative algebra.  The variety $\mathcal{V}(I_{(k)})$ consists of all $K$-linear maps from $K^m$ to $K^n$ of rank less than~$k$. 
Sturmfels~\cite{St90} and Caniglia et al.~\cite{CGG} showed that the set of all $k$-minors of $X$ forms a Gr\"obner basis of~$I_{(k)}$  (with respect to both diagonal and lexicographic orders). Moreover, in the case that  $k=\min\{m,n\}$ the generators of $I_{(k)}$ form a universal Gr\"obner basis (i.e., a
  Gr\"obner basis with respect to every term order on $K[X]$). Their technique provided a new proof of the Cohen--Macaulayness of $R/I_{(k)}$ in this case and has been used to compute numerical invariants of these rings, like the multiplicity and the Hilbert function, see~\cite{BH1,CH,HT}. Some excellent references on the theory of determinantal ideals are the book~\cite{BV} of Bruns and Vetter, and the paper~\cite{BC} of Bruns and Conca.

In general, for any $T\subseteq[n]$, we call $\Delta_{(k),T}:=\{ T'\subseteq T : |T'|=k\}$ the \emph{$k$-clique}
on~$T$.  Since the ideals $J_{\Delta_{(k),T}}$ are the easiest determinantal ideals, our strategy is to study
arbitrary hypergraphs in terms of their \emph{clique decompositions} (Section~\ref{sec:cliq-decomp}).

Several classes of determinantal ideals that have been studied in the literature turn out to be determinantal hypergraph ideals. For example, Herzog et al. in~\cite{HHH} and the second author in~\cite{Rauh13} studied the determinantal ideal $J_\Delta$ when $\Delta$ is a simple graph.
They showed that many of the algebraic properties of these ideals can be translated into combinatorial
properties of their corresponding graphs. These ideals arise naturally in the study of conditional independence statements, see~\cite{HH,HHH,FatemehLeila}.
It turns out that more general determinantal ideals also play a role for conditional independence statements in the presence of hidden variables~\cite{MR}.

Ho\c{s}ten
and Sullivant in~\cite{HSS} considered the ideal generated by maximal adjacent minors of a generic matrix. This ideal is the determinantal ideal of the \emph{adjacent hypergraph}~$\Delta_{(m)}^{\adj}$, i.e., the
hypergraph on $[n]$ with all hyperedges of the form $\{a,a+1\ldots,a+m-1\}$ for $1\leq a\leq n-m+1$.
The minimal primes of these ideals are again determinantal hypergraph ideals for suitable hypergraphs. The determinantal ideals of \emph{pure} hypergraphs (i.e.~hypergraphs where all hyperedges have the same cardinality) were studied in~\cite{EHHM}.

Motivated by geometrical considerations, the more general class of ladder determinantal ideals has been considered by Conca, Gonciulea, and Miller, see~\cite{A,mixed}. These ideals can be decomposed in a nice way as $J=J_{\Delta_1,S_1}+J_{\Delta_2,S_2}+\cdots+J_{\Delta_r,S_r}$ such that $J_{\Delta_i,S_i}$ are all classical determinantal ideals with
$S_r\subset \cdots\subset S_2\subset S_1$.
This decomposition helps to show how these ideals share many properties with the classical determinantal ideals, see~\cite{A,CH}.
In our paper, we find similar decompositions.

\medskip

\noindent{\bf Outline and our results.}
In Section~\ref{sec:det-hyp-ideals} we define when a hypergraph $\Delta$ is \emph{closed} and show that the generators
of the corresponding ideal $J_\Delta$ form a Gr\"obner basis with respect to the lexicographic order.  We compute the
numerical invariants of these ideals in Theorem~\ref{codim}.  Our
results 
imply that the generators of a determinantal ideal of a pure $(m-1)$-dimensional
hypergraph form a universal Gr\"obner basis if and only if the underlying hypergraph
  is a 
union of full sekeletons of simplices (that is, the corresponding ideal is a sum of classical determinantal ideals
$I_{(m-1)}$ defined in disjoint sets of variables); see Remark~\ref{rem:universal}.

In Section~\ref{sec:block-adjacent-hyps} we introduce \emph{block adjacent hypergraphs} and \emph{unions of subsequent
  block adjacent hypergraphs}.  Such hypergraphs can be written as a union $\Delta=\Delta_{1}\cup\dots\cup\Delta_{r}$,
where each $\Delta_i=\Delta_{(m),\{u_{i},u_{i}+1,\dots,v_{i}\}}$ is an $m$-clique on a set 
of consecutive vertices with~$u_1<u_2<\cdots<u_r$, $v_{1}<v_{2}<\cdots<v_{r}$ and $v_{i}-u_{i+1}<m$.  Their
determinantal ideals are generalizations of the ideal of maximal adjacent minors studied by Ho\c{s}ten and Sullivant
in~\cite{HSS}.  We give a combinatorial description of their minimal primes.  When $v_{i}-u_{i}=m-1$, then $\Delta$
contains~$\Delta^{\adj}_{(m)}$, and we call $\Delta$ \emph{block adjacent}.  In this case, we identify the set of
associated primes of $J_{\Delta}$ as a subset of the associated primes of~$J_{\smash[t]{\Delta_{(m)}^{\adj}}}$, see
Theorem~\ref{prime adjacent}.  On the other hand, when $v_{i}-u_{i}<m-1$, we call $\Delta$ a \emph{union of subsequent
  block adjacent hypergraphs}, and we show that a prime splitting occurs, see Theorem~\ref{thm:block-adj-prime-split}.

Our proofs are based on finding variables and minors that are regular (non-zero-divisor) modulo $J_{\Delta}$ and localizing with
respect to these variables and minors.  The main technical result is the prime splitting lemma (Lemma~\ref{lem:prime-splitting})
and the localization lemma (Lemma~\ref{rules}).  Further applications of the prime splitting lemma are studied in
Section~\ref{sec:further-examples}.  
We formulate what can be said using our methods about unions $\Delta=\Delta_{1}\cup\dots\cup\Delta_{r}$ of block
adjacent hypergraphs  or cliques
(Theorem~\ref{thm:lemma-as-thm}) 
under the assumption that the sub-hypergraphs $\Delta_{i}$ have a ``small'' and ``nice'' intersection.
%

Section~\ref{sec:res} is devoted to minimal free resolutions of the determinantal ideals associated to closed hypergraphs. The minimal free resolution of the classical  determinantal ideal generated by all maximal minors of $X$
is given by the Eagon-Northcott complex (see e.g.~\cite[A2.6.1]{Eisenbud}). In Theorem~\ref{res} we 
construct a minimal free resolution of a  determinantal hypergraph ideal.  The result is that the multigraded Betti numbers of these ideals are equal to the multigraded Betti numbers of their initial ideals with respect to the lexicographic term order, see Corollary~\ref{Betti}. This gives a positive answer to~\cite[Question~1.1]{ConcaHostenThomas}.

\medskip

\noindent{\bf Assumptions on the field.}
Later, we will often argue that a sum of prime ideals defined in disjoint sets of variables is prime.  In general, this
is only true over an algebraically closed field:
\begin{Example}
  \label{ex:disj-not-suff}
  If $K$ does not contain a root of $-1$, then $I_{1} = ( x^{2}+1 )$ and $I_{2} = (y^{2}+1)$ are both prime and defined
  in disjoint sets of variables.  However, the ideal $I = I_{1}+I_{2}$ has a non-trivial primary decomposition $I = (x +
  y, x^{2}+1) \cap (x - y, x^{2}+1)$.\qed
\end{Example}
We will not have this problem, since the ideals that we work with are, in fact, \emph{geometrically prime}, that is,
they remain prime under any algebraic field extension.  The following lemma shows that in our proofs we may pass to the
algebraic closure of $K$:
\begin{Lemma}
  \label{lem:field-lemma}
  Let $K\subseteq L$ be an algebraic field extension.  For any ideal $I\subseteq K[X]$ let $I_{L}\subseteq L[X]$ be the
  ideal generated by~$I$ in~$L[X]$.
  \begin{enumerate}
  \item If $I_{L}$ is prime, then $I$ is prime.
  \item If $f\in K[X]\setminus I$ is regular modulo~$I_{L}$, then $f$ is regular modulo~$I$.
  \end{enumerate}
\end{Lemma}
\begin{proof}
  An ideal $I$ is prime if and only if each $f\in K[X]\setminus I$ is regular modulo~$I$, so statement~(1)
  follows from statement~(2).  Let $g\in K[X]$, and assume that $fg\in I$.  Then $fg\in I_{L}$.  Since $f\notin I_{L}$ and
  since $I_{L}$ is prime, $g\in I_{L}$.  Hence $g\in I$.
\end{proof}

\section{Determinantal hypergraph ideals}
\label{sec:det-hyp-ideals}

\subsection{Clique decomposition \texorpdfstring{of $\Delta$}{}}
\label{sec:cliq-decomp}

Let $\Delta$ be a hypergraph on $[n]$.  A \emph{clique} or a \emph{$k$-clique} of
$\Delta$ is a sub-hypergraph of the form $\Delta_{(k),T}:=\{ T'\subseteq T : |T'|=k\}$ for some~$T\subseteq[n]$.  A
\emph{maximal $k$-clique} of $\Delta$ is a $k$-clique that is not a subset of another $k$-clique of $\Delta$ (i.e., it
is maximal with respect to inclusion).  Every hypergraph $\Delta$ can be written as the union of its maximal cliques
$\Delta_{1},\ldots,\Delta_{r}$.  The decomposition $\Delta=\Delta_{1}\union \Delta_{2}\union\cdots\union \Delta_{r}$ is
called the {\em clique decomposition} of $\Delta$.
Observe that the determinantal ideal of a clique is a classical determinantal ideal which is well-understood from the
algebraic point of view.
\begin{Example}
In examples, in order to simplify our notation we denote the set $F=\{i_1,\ldots,i_t\}$ by $i_1i_2\cdots i_t$,
where $i_1<i_2<\cdots<i_t$.
  Consider the hypergraph on the vertex set $[9]$ with hyperedges
  123, 124, 134, 234, 345, 567, 78, 89 and 79
  (see Figure~\ref{Fig:clique-decomposition:A}).  Its clique decomposition is
  $\Delta=\Delta_{(3),1234}\cup\Delta_{(3),345}\cup\Delta_{(3),567}\cup\Delta_{(2),789}$.
\end{Example}

\renewcommand{\thesubfigure}{\textnormal{(\alph{subfigure})}}

\begin{figure}[h!t]
  \begin{center}
    \subcaptionbox{\label{Fig:clique-decomposition:A}}{\begin{tikzpicture} [scale = .58, very thick = 10mm]
      \fill[clique] (-12,-1) rectangle (-10,1);
      \fill[clique] (-10,-1) -- (-10,1) -- (-8,0);
      \fill[clique] (-8,0) -- (-7,2) -- (-6,0);
      \node (n1) at (-12,-1)  [Cgray] {$1$};
      \node (n2) at (-12,1) [Cgray] {$2$};
      \node (n3) at (-10,-1)  [Cgray] {$3$};
      \node (n4) at (-10,1)  [Cgray] {$4$};
      \node (n5) at (-8,0)  [Cgray] {$5$};
      \node (n6) at (-7,2)  [Cgray] {$6$};
      \node (n7) at (-6,0)  [Cgray] {$7$};
      \node (n8) at (-5,2)  [Cgray] {$8$};
      \node (n9) at (-4,0)  [Cgray] {$9$};
      \foreach \from/\to in {n1/n2,n1/n3,n1/n4,n2/n3,n2/n4,n3/n4, n3/n5,n4/n5, n5/n6,n5/n7,n6/n7, n7/n8,n7/n9,n8/n9}
      \draw[] (\from) -- (\to);
    \end{tikzpicture}}
    $\qquad$
    \subcaptionbox{\label{Fig:clique-decomposition:B}}{\begin{tikzpicture} [scale = .58, very thick = 10mm]
      \fill[clique] (-12,-1) rectangle (-10,1);
      \fill[clique] (-10,-1) -- (-10,1) -- (-8,0);
      \fill[clique] (-8,0) -- (-7,2) -- (-6,0);
      \node (n1) at (-12,-1)  [Cgray] {$1$};
      \node (n2) at (-12,1) [Cgray] {$9$};
      \node (n3) at (-10,-1)  [Cgray] {$2$};
      \node (n4) at (-10,1)  [Cgray] {$8$};
      \node (n5) at (-8,0)  [Cgray] {$3$};
      \node (n6) at (-7,2)  [Cgray] {$7$};
      \node (n7) at (-6,0)  [Cgray] {$4$};
      \node (n8) at (-5,2)  [Cgray] {$6$};
      \node (n9) at (-4,0)  [Cgray] {$5$};
      \foreach \from/\to in {n1/n2,n1/n3,n1/n4,n2/n3,n2/n4,n3/n4, n3/n5,n4/n5, n5/n6,n5/n7,n6/n7, n7/n8,n7/n9,n8/n9}
      \draw[] (\from) -- (\to);
    \end{tikzpicture}}
    \caption{\subref*{Fig:clique-decomposition:A} A hypergraph with one- and two-dimensional hyperedges.  The color indicates which triangles have been
      added.  \subref*{Fig:clique-decomposition:B} A different hypergraph, isomorphic to the first one.}\label{Fig:clique-decomposition}
  \end{center}
\end{figure}
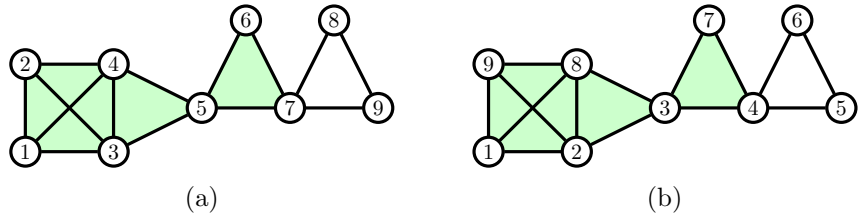

\subsection{Ideals whose generators form a Gr\"obner basis}
Let $<_{\lex}$ be the lexicographic order induced by the natural order of indeterminates
\[x_{11} > x_{12} > \cdots > x_{1n} > x_{21} > \cdots > x_{2n} > \cdots > x_{mn},\]
of matrix row by row from left to right.
Here we study the question when the generators of $J_\Delta$ form a Gr\"obner basis with respect to $<_{\lex}$.

\begin{Definition}
  \label{def:cl}
  Let $\Delta$ be a hypergraph on $[n]$ with clique decomposition $\Delta=\Delta_1\cup\cdots\cup\Delta_r$.  Then
  $\Delta$ is called \emph{closed} if minors belonging to different cliques have relatively prime initial terms (with
  respect to $<_{\lex}$).
\end{Definition}
Note that closedness depends on the value of~$m$.  A closed hypergraph becomes non-closed when~$m$ is increased.
Moreover, the definition of closedness depends on the ordering (or labeling) of the nodes, as the next example shows.
\begin{Example}
  Let $m=3$.  The hypergraph in Figure~\ref{Fig:clique-decomposition:A} is closed.  For example
  $\ini_<[123|134]=x_{11}x_{23}x_{34}$ and $\ini_<[123|345]=x_{13}x_{24}x_{35}$ are relatively prime; and
  $\ini_<[123|567]=x_{15}x_{26}x_{37}$ and $\ini_<[a_{1}a_{2}|78]=x_{a_17}x_{a_28}$ are relatively prime for all $1\leq
  a_1<a_2\leq 3$.  However, the isomorphic (relabelled) hypergraph in Figure~\ref{Fig:clique-decomposition:B} is not closed, since
  $\ini_<[123|128]=x_{11}x_{22}x_{38}$ and $\ini_<[123|238]=x_{12}x_{23}x_{38}$ are not relatively prime; and
  $\ini_<[123|347]=x_{13}x_{24}x_{37}$ and $\ini_<[23|46]=x_{24}x_{36}$ are not relatively prime either.
\end{Example}







\begin{Lemma}\label{lem:cl}
  Let $\Delta=\Delta_1\cup\cdots\cup\Delta_r$ be the clique decomposition of $\Delta$. Then $\Delta$ is closed if and
  only if for every pair of hyperedges $\{b_1,b_2,\ldots,b_t\}\in \Delta_i$ ($b_{1}<b_{2}<\dots<b_{t}$) and
  $\{c_1,c_2,\ldots,c_s\}\in \Delta_j$ ($c_{1}<c_{2}<\dots<c_{s}$) with $i\neq j$ and $t\leq s$, and for all $k$ we have
  \begin{equation*}
    b_k\neq c_\ell,\quad\text{where}\quad \max\{1,k-m+s\}\leq \ell\leq m-t+k.
  \end{equation*}
\end{Lemma}

\begin{proof}
For each hyperedge $\{b_1,b_2,\ldots,b_t\}$ of $\Delta$ with $b_{1}<b_{2}<\dots<b_{t}$, and integers $1\leq a_1<\cdots<a_t\leq m$, 
the initial term of $[a_1\ldots a_t| b_1\ldots b_t]$ is $x_{a_1b_1}\cdots x_{a_tb_t}$. Therefore
for each index $b_k$ the variable $x_{i b_k}$ appears in the support of an initial term of a minor $[a_1\ldots a_t| b_1\ldots b_t]$ if and only if $k\leq i\leq m-t+k$. The statement follows by this observation.
\end{proof}

\begin{Example}
  Consider again the hypergraph of Figure~\ref{Fig:clique-decomposition:B}, which is not closed.  For example, vertex 2
  takes the first position in two hyperedges 289 and 238, and vertex 4 takes the first position in the edge 46 and the
  second position in the hyperedge~347.
\end{Example}

\begin{Proposition}\label{closed}
  If $\Delta$ is a closed hypergraph, then the generators of $J_\Delta$ form a $<_{\lex}$-Gr\"obner basis,
  and $J_\Delta$ is a radical ideal.
\end{Proposition}
\begin{proof}
We show that all $S$-pairs, $S([a_1\ldots a_t|b_1\ldots b_t], [c_1\ldots c_s|d_1\ldots d_s])$ reduce
to zero. Assume that $\{b_1,\ldots,b_t\}\in \Delta_i$ and $\{d_1,\ldots,d_s\}\in \Delta_j$. If $i\neq j$, then $\text{in}_<[a_1\ldots a_t|b_1\ldots b_t]$ and $\text{in}_< [c_1\ldots c_s|d_1\ldots d_s]$ have no common factor, which implies that their $S$-polynomial reduces to zero.
Now assume that $i=j$. Then $s=t$, and all $s$-subsets of $\{b_1,\ldots,b_s\}\cup\{d_1,\ldots,d_s\}$ belong to $\Delta_{i}$. Therefore, since the set of all $s$-minors in $J_{\Delta_{i}}$ forms a Gr\"obner basis of $J_{\Delta_{i}}$, the $S$-pair $S([a_1\ldots a_s|b_1\ldots b_s], [c_1\ldots c_s|d_1\ldots d_s])$ reduces to zero with respect to
the $s$-minors of~$J_{\Delta_{i}}$. Thus the assertion follows from Buchberger's criterion. Since the generators of the initial ideal of $J_\Delta$ are all squarefree, we deduce that $J_\Delta$ is radical.
\end{proof}
\begin{Remark}
\begin{itemize}
\item[(1)] In fact, if $\Delta$ is a closed hypergraph, then the generators of $J_\Delta$
form a minimal Gr\"obner basis for $J_\Delta$ with respect to any {diagonal} term order in which $x_{ij}>x_{ik}$ for all $i$ and for all $j<k$.
Here we always consider the lexicographic term order $<$ induced by $x_{11}>\cdots>x_{1n}>\cdots>x_{mn}$. 
\item[(2)] 
When $\Delta$ is a pure $(m-1)$-dimensional hypergraph, then the converse of the first statement in Proposition~\ref{closed} also holds, as shown in~\cite[Theorem~1.1]{EHHM}: In this case, the generators of $J_{\Delta}$ form a $<_{\lex}$-Gr\"obner basis if and only  if $\Delta$ is closed.  Theorem~\ref{thm:glue-along-cliques} below shows that same is not true more generally: There are non-closed mixed hyper graphs $\Delta$ such that the generators of $J_{\Delta}$ form a Gr\"obner basis.
\end{itemize}
\end{Remark}

A consequence of Proposition~\ref{closed} is the following result which follows by the same argument as in the proof of~\cite[Corollary~1.3]{EHHM}.
\begin{Theorem}
\label{codim}
Let $\Delta$ be a closed hypergraph with clique decomposition $\Delta=\Delta_1\cup\cdots\cup\Delta_r$.
\begin{itemize}
\item [(1)] $\displaystyle\height J_\Delta = \sum_{\ell=1}^r \height\ J_{\Delta_\ell}=\sum_{\ell=1}^r (|V(\Delta_\ell)|-\dim(\Delta_{\ell}))$.
\item [(2)] $J_\Delta$ is Cohen-Macaulay.
\item [(3)] The Hilbert series of $R/J_\Delta$ has the form
  $\displaystyle
	H_{R/J_\Delta}(t)=\prod_{\ell=1}^r H_{R/J_{\Delta_\ell}}(t)$.
	\item [(4)] The multiplicity of $R/J_\Delta$ is
	\[
	e(R/J_\Delta)=\prod_{\ell=1}^r {e(R/J_{\Delta_\ell})}=\prod_{\ell=1}^r \binom{|V(\Delta_\ell)|}{\dim(\Delta_{\ell})}.
	\]
\end{itemize}
\end{Theorem}

\begin{proof}
By assumption, there are polynomial rings $R_{\ell}$ in disjoint sets of variables, such that the initial ideals ${\ini}_<(J_{\Delta_\ell})$ are generated by monomials in $R_{\ell}$ and such that
\begin{equation}
\label{tensor}
R/\ini_<(J_\Delta)\cong \bigotimes_{\ell=1}^r R_\ell/\ini_<(J_{\Delta_\ell}).
\end{equation}
This fact together with the known formula for the height of the
determinantal ideals (see e.g.~\cite[Theorem 6.35]{EH}) implies (1). It is known that $R_\ell/\ini_<(J_{\Delta_\ell})$ is Cohen-Macaulay for each $\ell$
(see e.g.~\cite{CGG} and~\cite{St90}). Therefore, by (\ref{tensor}), we get that $R/\ini_<(J_\Delta)$ is Cohen-Macaulay, and so
$R/J_\Delta$ is Cohen--Macaulay as well (see e.g.~\cite[Corollary 3.3.5]{HHBook}).

By~\cite[Corollary 3.3.5]{HHBook}, the modules $R/J_{\Delta}$ and $R/\ini_<(J_\Delta)$ have the same Hilbert series. The Hilbert series and the multiplicity of the classical determinantal rings generated by maximal minors of $X$ are known (see e.g.~\cite[Corollary 1]{CH} or
\cite[Theorem 6.9]{BC} and~\cite[Theorem 3.5]{HT}). Therefore equation (\ref{tensor}) implies statements (3) and (4).
\end{proof}

\begin{Remark}
  \label{rem:universal}
  Let $\Delta$ be a {pure $(m-1)$-dimensional}
  hypergraph on the vertex set~$[n]$  (i.e.~all elements of~$\Delta$ have cardinality~$m$). By  {~\cite[Theorem~1.1]{EHHM},} 
  the generators of $J_\Delta$ form a Gr\"obner basis with respect to the lexicographic order induced by
  $x_{11}>\cdots>x_{1n}>\cdots>x_{mn}$ if and only if $\Delta$ is a closed hypergraph.
  It is easy to see that a hypergraph $\Delta$ is closed with respect to any renumbering of its nodes if and only if
  different maximal cliques in its clique decomposition have disjoint vertices; i.e., $\Delta$ is a disjoint union of
  cliques.
  Hence, we conclude that the generators of $J_\Delta$ form a Gr\"obner basis with respect to the lexicographic order
  induced by every term order on variables $x_{ij}$, if and only if $\Delta$ is a full skeleton of a simplex,
  where, for the converse statement, we need that the generators of the classical maximal determinantal ideal $I_{(k)}$ indeed form a
  Gr\"obner basis with respect to any lexicographic term order; see \cite{BZ,SturmfelsZelevinsky}.
As shown in~\cite{CGG,St90}, the generators of $I_{(k)}$ form a  Gr\"obner basis with respect to both diagonal and lexicographic orders.  In the case that  $k=\min\{m,n\}$ the generators of $I_{(k)}$ even form a universal Gr\"obner basis.
Therefore, if $\Delta$ is a pure $(m-1)$-dimensional
  hypergraph, then the generators of $J_{\Delta}$ form a universal
  Gr\"obner basis of $J_\Delta$ if and only if $\Delta$ is a disjoint union of $m$-cliques.
\end{Remark}

\subsection{Gr\"obner bases of a class of mixed determinantal ideals}
Next, we study a class of mixed determinantal ideals $J_{\Delta}$ which enjoy the property that the minors that generate $J_\Delta$ form a $<_{\lex}$-Gr\"obner basis.
The minimal primes of the determinantal ideals studied in Sections~\ref{sec:block-adjacent-hyps} and~\ref{sec:further-examples} will be of this type.

We first fix our notation.
For any pair of index sets $S=\{i_1<\cdots<i_k\}\subseteq [m]$ and $T=\{j_1<\cdots <j_t\}\subseteq [n]$, we denote by $X_{T}^S$ the submatrix of $X$ with row indices $i_1,\ldots,i_k$, and column indices $j_1,\ldots,j_t$.
{If $S=[m]$, then we simplify the notation to $X_T$.}

%

In order to prove our main result of this section, we need the following technical lemma, which is a slight generalization of
\cite[Lemma~4.2]{HSS}.

\begin{Lemma}
  \label{GB1}
  Let $k\le m$, and assume that one of the following holds true:
  \begin{enumerate}
  \item $T=[k]$, and $S'\subseteq[m]$ contains $T$; or
  \item $T=\{n-k+1,\dots,n\}$, and $S'\subseteq[m]$ contains $\{m-k+1,\dots,m\}$. 
  \end{enumerate}
  Let $G$ be the set of $k$-minors of~$X^{[m]}_{T}$, and let $G'$ be the set of $k'$-minors of~$X^{S'}_{[n]}$, where
  $k'$ is the cardinality of~$S'$.  Then $G\cup G'$ forms a $<_{\lex}$-Gr\"obner basis for the ideal it generates.
\end{Lemma}

\begin{proof}
  We assume that $T=[k]$.  The other case follows by a similar argument.

The set $G$ is a Gr\"obner basis for the ideal $I=( G )$ it generates, and $G'$ is a Gr\"obner basis for $J= ( G')$.
Let $T=[k]$ and $S'=\{c_1,\ldots,c_{k'}\}$ with $T\subset S'$.
By \cite[Lemma 1.3(c)]{Aldo},
in order to prove the statement, we have to show that for arbitrary minors
$f=[a_1\ldots a_k|1\ldots k]\in G$ and $g=[c_1\ldots c_{k'}|b_1\ldots b_{k'}]\in G'$
there exists an element $h\in I\cap J$ with $\ini(h) = \text{LCM}(\ini(f), \ini(g))$.
To prove this statement, we construct a matrix
\[H=
\left(
  \begin{array}{cccc}
    W & 0 \\
    Y & Z \\
     \end{array}
\right)
\]
with the following properties:
\begin{enumerate}
\item The initial term of the determinant $|H|$ is equal to $\text{LCM}(\ini(f), \ini(g))$.
\item $Y$ and $Z$ are submatrices of $X_{[n]}^{S'}$ with $k'$ rows.
\item ${W}$ and ${Y}$ are submatrices of $X_{T}^{[m]}$ with $k$ columns.
\end{enumerate}
(1)~implies that the initial term of $h:=|H|$ equals $\ini(h) = \text{LCM}(\ini(f), \ini(g))$.  Properties (2)~and
(3)~imply that $h\in I\cap J$.  Indeed, computing the Laplace expansion of $|H|$ starting with the last $k$ rows shows
that $h$ can be written as a polynomial combination of $k$-minors of~$X_{T}^{[m]}$.  Similarly, computing the Laplace
expansion of $|H|$ starting with the first $k'$ columns shows $h\in J$.

Let $A=\{a_{1},\dots,a_{k}\}$ and $B=\{b_{1},\dots,b_{k'}\}$.
Let $B_0$ be the subset of columns of $X_{[n]}^{S'}$ indexed by $b_j$ with $x_{a_\ell,\ell}=x_{c_j,b_{j}}$ for some
$\ell$. 
Assume that $B\backslash B_{0}=\{b_{j_{1}},\dots,b_{j_{s}}\}$. Similarly assume that $A\backslash A_{0}=\{a_{\ell_{1}},\dots,a_{\ell_{s}}\}$ consists of the subset of rows of $X_T^{A}$  indexed by $a_{\ell_i}$, where $\ell_i\not\in B_0$.
Then we set $Z = X^{S'}_{B\backslash B_{0}}$ and
 $W=X_{T}^{A\backslash A_0}$. 
Finally, let $Y=X_{T}^{S'}$.

Properties (2) and (3) are easy to check.
To see (1), observe that $\text{LCM}(\ini(f), \ini(g))$ appears as a term in the expansion of $|H|$.  Then (1) follows, as in the proof of~\cite[Lemma~4.2]{HSS},
since if $|H|$ would contain a larger term, then either $f$ or $g$ would contain a term larger than $\ini(f)$ or
$\ini(g)$, respectively.
\end{proof}

The next example illustrates the idea of the proof of Lemma~\ref{GB1}.
\begin{Example}
Let $k=4$, $k'=5$, $T=\{1,2,3,4\}$ and  $S'=\{1,2,3,4,6\}$.
For $f=[1367|1234]\in F$ and $g=[12346|12567]\in G$, we have $\ini(f)=x_{11}x_{32}x_{63}x_{74}$ and $\ini(g)=x_{11}x_{22}x_{35}x_{46}x_{67}$.
Then $B_0=\{1\}$, and
\newcommand\bovermat[2]{%
  \makebox[0pt][l]{$\smash{\overbrace{\phantom{%
    \begin{matrix}#2\end{matrix}}}^{\text{#1}}}$}#2}
\newcommand\bundermat[2]{%
  \makebox[0pt][l]{$\smash{\underbrace{\phantom{%
    \begin{matrix}#2\end{matrix}}}_{\text{#1}}}$}#2}
\begin{equation*}
  H= 
    \arraycolsep=2pt
    \begin{pmatrix}
      \arraycolsep=5pt
      \bovermat{W}{\fbox{$\begin{matrix}
          x_{31} & \underline{{x_{32}}} & x_{33} & x_{34} \\
          x_{61} &  x_{62} & \underline{x_{63}} & x_{64}  \\
          x_{71} &  x_{72} & x_{73} & \underline{x_{74}}
        \end{matrix}$}}
      &
      \arraycolsep=5pt
      \phantom{00}
      \begin{matrix}
        \rlap{0}\phantom{x_{31}} &       \rlap{0}\phantom{x_{31}} & \rlap{0}\phantom{x_{31}} & \rlap{0}\phantom{x_{31}} \\
        \rlap{0}\phantom{x_{31}} &       \rlap{0}\phantom{x_{31}} & \rlap{0}\phantom{x_{31}} & \rlap{0}\phantom{x_{31}} \\
        \rlap{0}\phantom{x_{31}} &       \rlap{0}\phantom{x_{31}} & \rlap{0}\phantom{x_{31}} & \rlap{0}\phantom{x_{31}} \\
      \end{matrix}
      \vspace{3pt}
      \\
      \arraycolsep=5pt
      \bundermat{Y}{\fbox{$\begin{matrix}
          \underline{x_{11}} & {x_{12}} & x_{13} & x_{14} \\
          x_{21} & {x_{22}} & x_{23}  & x_{24}   \\
          x_{31} & {x_{32}} & x_{33}  & x_{34}   \\
          x_{41} & {x_{42}} & x_{43}  & x_{44}   \\
          x_{61} & {x_{62}} & x_{63} & x_{64}
        \end{matrix}$}}
      &
      \arraycolsep=5pt
      \bundermat{Z}{\fbox{$\begin{matrix}
          {x_{12}} & x_{15}  & x_{16} & x_{17}  \\
          \underline{x_{22}} & x_{25}  & x_{26}  & x_{27} \\
          {x_{32}} & \underline{x_{35}}  & x_{36}  & x_{37} \\
          {x_{42}} & x_{45}  & \underline{x_{46}} & x_{47}  \\
          {x_{62}} & x_{65}  & x_{66}  & \underline{x_{67}}  \\
        \end{matrix}$}}
    \end{pmatrix}
    .
\end{equation*}
\vspace{15pt}

\noindent%
It is easy to see that $\ini(h) = {\rm LCM}(\ini(f), \ini(g))$ (which is the product of the underlined entries of the matrix $H$) using the Laplace expansion along
the first four columns
of
the matrix $H$.
\end{Example}


Lemma~\ref{GB1} can be applied iteratively to hypergraphs that are obtained by gluing cliques along their intersections as follows.
\begin{Theorem}
  \label{thm:glue-along-cliques}
  Let $\Delta_{1},\dots,\Delta_{r}$ be $m$-cliques on~$[n]$, and let $\Delta_{i,j}=\{V(\Delta_{i})\cap
  V(\Delta_{j})\}$ (a hypergraph with a single hyperedge).  Assume that the following conditions are satisfied:
  \begin{enumerate}
  \item $k_{i,j}:=|V(\Delta_{i,j})|< m$ for all~$i\neq j$.
  \item For all~$i\neq j$, the set $V(\Delta_{i,j})$ consists of either the first or the last $k_{i,j}$ columns of
    $V(\Delta_{i})$.
  \item If $i\neq j\neq k\neq i$, then $V(\Delta_{i})\cap V(\Delta_{j})\cap V(\Delta_{k})=\emptyset$.
  \end{enumerate}
  Then the set of maximal minors of the ideal $J_{\Delta}$ of $\Delta=\bigcup_{i}\Delta_{i}\cup\bigcup_{i\neq
    j}\Delta_{i,j}$ forms a $<_{\lex}$-Gr\"obner basis for~$J_{\Delta}$. In particular, $I$ is a radical ideal.
\end{Theorem}
\begin{proof}
  The assumptions imply that for all $i\neq j$, the initial terms of elements of $J_{\Delta_i}$ and of elements
  of $J_{\Delta_j}$ are relatively prime, and so $S$-pairs constructed from elments of $J_{\Delta_{i}}$ and
  $J_{\Delta_{j}}$ reduce to zero.
  By Lemma~\ref{GB1}, any $S$-pair of an element of $J_{\Delta_{i,j}}$ with an element of $J_{\Delta_{i}}$ reduces to
  zero.  Finally, by~(3), the initial terms of elements of $J_{\Delta_{i,j}}$ and elements of $J_{\Delta_{k,\ell}}$ are
  relatively prime, unless $\{i,j\}=\{k,\ell\}$, and so their $S$-pairs also reduce to zero.  This implies the first
  statement. Since the initial terms of the Gr\"obner basis are all squarefree, $I$ is radical.
\end{proof}

\section{Determinantal ideals of block adjacent hypergraphs}
\label{sec:block-adjacent-hyps}

Let $\Delta={\Delta_1}\cup\cdots\cup{\Delta_r}$, where each $\Delta_i=\Delta_{(m),\{u_{i},u_{i}+1,\dots,v_{i}\}}$ is an $m$-clique on
$\{u_i,u_i+1,\ldots,v_i\}$ such that $u_1<u_2<\cdots<u_r$ and $V(\Delta_{i-1})\cap
V(\Delta_{i})=\{u_{i},u_{i}+1,\ldots,u_{i}+t_i-1\}$ for some $0\le t_i<m$. Observe that $\Delta$ is closed.  The
determinantal ideal $J_{\Delta}$ generalizes the ideal $J_{\smash[t]{\Delta_{(m)}^{\adj}}}$ of maximal adjacent minors studied by
Ho\c{s}ten and Sullivant in~\cite{HSS}, where $\smash[t]{\Delta_{(m)}^{\adj}}$ is the hypergraph on $[n]$ with all hyperedges of
the form $\{a,a+1\ldots,a+m-1\}$ for $1\leq a\leq n-m+1$.
We first consider the case where $t_i=m-1$ for all $i$, i.e., the vertex sets of successive
cliques intersect in $m-1$ vertices.   In this case, $J_{\Delta}$ contains
$J_{\smash[t]{\Delta_{(m)}^{\adj}}}$, 
and the hypergraph is called \emph{block adjacent}.


\subsection{Block adjacent hypergraphs and prime sequences}
\label{sec:block-adjacent-prime-seqs}

Let $\Delta=\Delta_1\cup\cdots\cup \Delta_r$ be a block adjacent hypergraph on the vertex set $[n]$ such that
$V(\Delta_{i-1})\cap V(\Delta_{i})=\{u_{i},u_{i}+1,\ldots,u_{i}+m-2\}$ for all $i$.  To describe the minimal primes for
$J_\Delta$ in terms of other determinantal hypergraph ideals, we generalize the definition of a prime sequence
from~\cite{HSS}.
\begin{Definition}
\label{def:prime-sequence:1}
A \emph{prime sequence} of $\Delta$ is a sequence of intervals
\[
 \Gamma:\ \ [a_{1}, b_{1}],[a_{2}, b_{2}],\ldots,[a_{t}, b_{t}]
\]
with the following properties:
\begin{itemize}
\item [(1)] $1=a_{1}<a_{2}<\cdots<a_{t}$ and $b_{1} < b_{2}<\cdots<b_{t}=n$;

\item [(2)] $b_{\ell}-a_{\ell}\geq  m-1$ for $\ell=1,t$ and $b_{\ell}-a_{\ell}\geq  m$ for $2\leq \ell\leq t-1$;

\item [(3)] $0 \leq b_{\ell}- a_{\ell+1} \leq m - 2$ for all $\ell$;

\item [(4)] if $|V(\Delta_i)|>m$ for some $i$, then there exists $\ell$ with $V(\Delta_i)\subseteq [a_{\ell}, b_{\ell}]$.
\end{itemize}
We denote the set of all prime sequences of $\Delta$ by $\mathcal{A}_\Delta$.

To each prime sequence $\Gamma$ we associate the ideal $P_{\Gamma}$ of $R$ generated by
\begin{itemize}
 \item[(i)] all maximal $m$-minors of the submatrix $X_{[a_{\ell},b_{\ell}]}^{[m]}$ for $\ell=1,\ldots,t$, and

 \item[(ii)] all  maximal $(b_{\ell}-a_{\ell+1}+1)$-minors of the submatrix $X_{[a_{\ell+1}, b_{\ell}]}^{[m]}$ for $\ell=1,\ldots,t-1$.
\end{itemize}
Observe that $P_\Gamma=J_{\Delta_\Gamma}$, where $\Delta_\Gamma$ is the union of the $m$-cliques 
on the vertex sets
$[a_i,b_i]$ and the $(b_{\ell}-a_{\ell+1}+1)$-cliques on $[a_{\ell+1}, b_{\ell}]$.
If we need to restrict the row indices of the defining minors, we write $P_\Gamma^B:=J_{\Delta_{\Gamma}}^{B}$.

\end{Definition}

\begin{Example}
  Let $m=3$.
  The hypergraph from Figure~\ref{Fig:block-adjacent-2} has the clique decomposition
  $\Delta_{(3),1234}\cup\{345\}\cup\{456\}\cup\{567\}$.
  Therefore, $\Delta$ is block adjacent.
  It has the following seven prime sequences:
  \begin{align*}
    \Gamma:&\ [1,7]         &
    \Gamma:&\ [1,6],[5,7]   &
    \Gamma:&\ [1,5], [5,7]  \\
    \Gamma:&\ [1,5],[4,7]   &
    \Gamma:&\ [1,4], [4,7]  &
    \Gamma:&\ [1,4], [3,7]  \\
    \Gamma:&\ [1,4], [3,6],[5,7].
  \end{align*}
\end{Example}

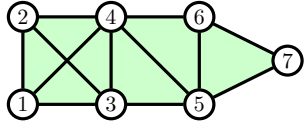
\begin{figure}[ht]
  \begin{center}
    \begin{tikzpicture} [scale = .58, very thick = 10mm]
      \fill[clique] (-12,-1) -- (-10,-1) -- (-8,-1) -- (-6,0) -- (-8,1) -- (-10,1) -- (-12,1);
      \node (n1) at (-12,-1)  [Cgray] {$1$};
      \node (n2) at (-12,1)  [Cgray] {$2$};
      \node (n3) at (-10,-1)  [Cgray] {$3$};
      \node (n4) at (-10,1)  [Cgray] {$4$};
      \node (n5) at (-8,-1)  [Cgray] {$5$};
      \node (n6) at (-8,1)  [Cgray] {$6$};
      \node (n7) at (-6,0)  [Cgray] {$7$};
      \foreach \from/\to in {n1/n2,n1/n3,n1/n4,n2/n3,n2/n4,n3/n4, n3/n5,n4/n5, n4/n6,n5/n6, n5/n7,n6/n7}
      \draw[] (\from) -- (\to);
    \end{tikzpicture}
  \end{center}
  \caption{A block adjacent hypergraph.}\label{Fig:block-adjacent-2}
\end{figure}

We first fix our notation.
For $1\le a< b\le n$, let $X[a,b]:=X_{\{a,a+1,\dots,b\}}$ be the submatrix of $X$ of those columns with indices
$a,a+1,\dots,b$.
Denote by $\spn[a, b]$ the vector space generated by the columns of $X[a,b]$. The maximum number of linearly independent columns of $X[a,b]$, i.e., rank of $X[a,b]$, is denoted by $\rk[a,b]$.

\begin{Theorem}
\label{prime adjacent}
Let $\Delta=\Delta_1\cup\cdots\cup\Delta_r$ be a block adjacent hypergraph.
The minimal primary decomposition of $J_\Delta$ is given by
$J_\Delta= \bigcap_{\Gamma\in \mathcal{A}_\Delta} P_\Gamma.$
\end{Theorem}

\begin{proof}
Any block adjacent hypergraph is closed, and so, by Proposition~\ref{closed}, the ideal $J_\Delta$ is radical.
Therefore, by the Nullstellensatz, it is enough to show that $\mathcal{V}(J_\Delta)= \bigcup_{\Gamma\in \mathcal{A}_\Delta} \mathcal{V}(P_\Gamma)$.
If $|V(\Delta_i)|=m$ for all~$i$, then $\Delta=\Delta^{\adj}_{(m)}$ is an adjacent hypergraph, and the statement holds by~\cite[Theorem~3.5]{HSS}.  Assume that there exists at least one clique with $|V(\Delta_\ell)|>m$.
Let $X$ be a matrix in $\mathcal{V}(J_\Delta)$. Note that $\mathcal{V}(J_\Delta)$ is a subset of $\mathcal{V}(J_{\smash[t]{\Delta^{\adj}_{(m)}}})$.
By~\cite[Lemma~3.4]{HSS} there exists a sequence $\Gamma:\ [a_1,b_1],\ldots,[a_k,b_k]$
such that $X\in \mathcal{V}(P_\Gamma)$ and $\Gamma$ has properties (1), (2) and (3). Assume that
$\Delta_{i_1},\ldots,\Delta_{i_p}$ are the cliques of $\Delta$
such that $|V(\Delta_{i_j})|>m$ and $V(\Delta_{i_j})$ is not the subset of any interval of $\Gamma$.
We will construct a sequence $\Gamma'$ with properties (1), (2), and (3) such that $X\in \V(P_{\Gamma'})$,
$V(\Delta_{i_1})$ is the subset of an interval of $\Gamma'$ and \emph{just} for cliques $\Delta_{i_2},\ldots,\Delta_{i_p}$,
the vertex set is not contained in any interval of $\Gamma'$.  Iterating this argument we obtain a prime sequence $\Gamma'$ for $\Delta$ such that $X\in \mathcal{V}(P_{\Gamma})$. 
The minimality of the constructed prime ideals follows by Corollary~4.5 and Corollary 4.6 of~\cite{HSS},
which completes the proof.

\medskip

Assume that $V(\Delta_{i_1})=[u,v]$.
Therefore there exist integers $s$ and $t$ such that
\[[u,v]\subseteq [a_{t},b_t]\cup [a_{t+1},b_{t+1}]\cup\cdots\cup [a_s,b_s],\]
where $a_t<u<b_t$ or $a_{t+1}=b_{t}=u$, and $a_s <v< b_{s}$ or $b_{s-1}=a_{s}=v$.
We consider four cases:

\medskip

Case 1.  If $v-a_s+1<m$ and $b_{t}-u+1<m$, then we consider the sequence
\[
\Gamma_1:\ \ [a_1,b_1],\ldots,[a_{t},b_t],[u,v],[a_s,b_s],\ldots,[a_k,b_k].
\]
Note that all $m$-minors of $X[u,v]$ are zero, since $X\in \V(J_\Delta)\subset \V(J_{\Delta_\ell})$.
Also all maximal minors of $X[u,b_t]$ and $X[a_s,v]$ are zero, since $[a_{t+1},b_t]\subseteq [u,b_t]$, $[a_s,b_{s-1}]\subseteq [a_{s},v]$ and by our assumption that $X\in P_\Gamma$ we have all maximal minors of $X[a_{t+1},b_t]$ and $X[a_s,b_{s-1}]$ are zero.
Therefore $X\in \mathcal{V}(P_{\Gamma_1})$, as desired.

\medskip
Case 2. Let $v-a_s+1\geq m$ and $b_{t}-u+1< m$. If $v=n$, then the sequence
\[
\Gamma':\ \ [a_1,b_1],\ldots,[a_{t},b_t],[u,v]
\]
has desired properties by the same argument as Case 1.
Assume that $v<n$.
If $\rk[a_s,v]=m-1$, then $\spn[a_s,b_s]=\spn[a_s,v]=\spn[u,v]=\spn[u,b_s]$.
This shows that all $m$-minors of $X[u,b_s]$ are zero. Moreover,
$[a_{t+1},b_t]\subseteq [u,b_t]$ implies that all maximal minors of $X[u,b_t]$ are zero, since $X\in \V(P_\Gamma)$. Thus the sequence
\[
\Gamma_2:\ \ [a_1,b_1],\ldots,[a_{t},b_t],[u,b_s],[a_{s+1},b_{s+1}],\ldots,[a_k,b_k]
\]
has desired properties.

\medskip

Assume that $\rk[a_s,v]<m-1$. If $b_s>v+1$, then we define $\Gamma'_2$ as
\[
\Gamma'_2:\ \ [a_1,b_1],\ldots,[a_{t},b_t],[u,v],[v-(m-2),b_s],[a_{s+1},b_{s+1}],\ldots,[a_k,b_k].
\]
Note that the width of the interval $[v-(m-2),b_s]$ is greater than $m$ and the width of the interval $[v-(m-2),v]$, i.e., the intersection of intervals
$[u,v]$ and $[v-(m-2),b_s]$ is $m-1$. These facts, together with the above condition on $\rk[a_s,v]$ show that the constructed sequence $\Gamma_2'$ has desired properties.

\medskip

Now assume that $b_s=v+1$. If $b_s=n$, then $\Gamma'_2$ has desired properties. Otherwise $[a_{s+1},b_{s+1}]$ is among the intervals $\Gamma$.
Now we should consider two different subcases:

Subcase 2.1. Let $a_{s+1}<v$. If $\rk[a_{s+1},v]<v-a_{s+1}+1$, then the sequence
\[
\Gamma_{2.1}:\ \ [a_1,b_1],\ldots,[a_{t},b_t],[u,v],[a_{s+1},b_{s+1}],[a_{s+2},b_{s+2}],\ldots,[a_k,b_k].
\]
fulfills our conditions. If $\rk[a_{s+1},v]=v-a_{s+1}+1$, then $v+1^{\rm th}$ column belongs to $\spn[a_{s+1},v]$ which is the subset of $\spn[u,v]$. Therefore
all $m$-minors of $X[u,v+1]$ are zero which implies that the following sequence fulfills our conditions:
\[
\Gamma_{2.2}:\ \ [a_1,b_1],\ldots,[a_{t},b_t],[u,v+1],[a_{s+1},b_{s+1}],[a_{s+2},b_{s+2}],\ldots,[a_k,b_k].
\]

Subcase 2.2.
Let $a_{s+1}\geq v$.
If $a_{s+1}=v+1$, then the variables of the $v+1^{\rm th}$ column of $X$ are in $P_\Gamma$, and so all $m$-minors of $X[u,b_s]$ are zero. Hence, the sequence $\Gamma_{2.2}$ has again desired properties.
Now, assume that $a_{s+1}=v$. If the $v^{\rm th}$ column of $X$ is nonzero, then $v+1^{\rm th}$ column is a multiplication of the $v^{\rm th}$ column and so it belongs to $\spn[u,v]$, since $\rk[v,v+1]=1$. Therefore $\Gamma_{2.2}$ fulfills our conditions.
Otherwise, the following sequence has desired properties:
\[
\Gamma_{2.3}:\ \ [a_1,b_1],\ldots,[a_{t},b_t],[u,v],[v,b_{s+1}],[a_{s+2},b_{s+2}],\ldots,[a_k,b_k].
\]

\medskip
Case 3. Let $v-a_s+1<m$ and $b_{t}-u+1\geq m$.
If $u=1$, then by the same argument as Case 1 the sequence
\[
 \Gamma':\ \ [u,v],[a_s,b_s],\ldots,[a_k,b_k]
\]
has desired property.
Suppose that $u>1$.
If $\rk[u,b_t]=m-1$, then $\spn[a_t,b_t]=\spn[u,b_t]=\spn[u,v]=\spn[a_t,v]$.
Then $\rk[a_t,v]<m$ shows that all $m$-minors of $X[a_t,v]$ are zero, and so the following sequence has desired properties:
\[
\Gamma_3:\ \ [a_1,b_1],\ldots,[a_{t-1},b_{t-1}],[a_{t},v],[a_s,b_s],\ldots,[a_k,b_k].
\]

Now assume that $\rk[u,b_t]<m-1$. If $a_t<u-1$, then we consider the sequence
\[
\Gamma_3':\ \ [a_1,b_1],\ldots,[a_{t},u+(m-2)],[u,v],[a_s,b_s],[a_{s+1},b_{s+1}],\ldots,[a_k,b_k].
\]
The width of the interval $[a_{t},u+(m-2)]$ is greater than $m$, and the width of $[u,u+(m-2)]$, i.e., the intersection of the intervals $[u,v]$ and $[a_{t},u+(m-2)]$ is $m-1$. Hence, our condition on $\rk[u,b_t]$ guarantees that $\Gamma'_3$ fulfills desired properties.

\medskip

Let $a_t=u-1$. If
$a_t=1$, then $\Gamma_3'$ has desired properties.
Assume that $a_t=u-1>1$. So we have $[a_{t-1},b_{t-1}]\in \Gamma$. Now two different subcases should be considered:

Subcase 3.1. Let $b_{t-1}>u$. If $\rk[u,b_{t-1}]<b_{t-1}-u+1$, then the sequence
\[
\Gamma_{3.1}:\ \ [a_1,b_1],\ldots,[a_{t-1},b_{t-1}],[u,v],[a_{s},b_{s}],\ldots,[a_k,b_k]
\]
has desired properties. If $\rk[u,b_{t-1}]=b_{t-1}-u+1$, then we consider the sequence
\[
\Gamma_{3.2}:\ \ [a_1,b_1],\ldots,[a_{t-1},b_{t-1}],[u-1,v],[a_{s},b_{s}],\ldots,[a_k,b_k].
\]
Note that $u-1^{\rm th}$ column belongs to $\spn[u,b_{t-1}]$. This together with the fact that
$\spn[u,b_{t-1}]\subseteq\spn[u,v]$ implies that the $u-1^{\rm th}$ column belongs to
$\spn[u,v]$ and so all $m$-minors of $X[u-1,v]$ are zero.

Subcase 3.2. Let $b_{t-1}\leq u$. If $b_{t-1}=u-1$, then the variables corresponding to the $u-1^{\rm th}$ column of $X$ are all in $P_\Gamma$,
and so all $m$-minors of $X[a_{t-1},v]$ are zero. Hence, the sequence $\Gamma_{3.2}$ has desired properties.

Let $b_{t-1}=u$. If the $u^{\rm th}$ column of $X$ is nonzero, then $u-1^{\rm th}$ column is a multiplication of the $u^{\rm th}$ column
and so it belongs to $\spn[u,v]$, since $\rk[u-1,u]=1$. Therefore $\Gamma_{3.2}$ has desired properties.
Otherwise, the sequence
\[
\Gamma_{3.3}:\ \ [a_1,b_1],\ldots,[a_{t-1},b_{t-1}],[u,v],[a_s,b_{s}],\ldots,[a_k,b_k]
\]
has desired properties.

\medskip

Case 4. If $v-a_s+1\geq m$ and $b_{t}-u+1\geq m$, then by combination of the arguments given in Case 2 and Case 3, we can construct the proper prime sequences.
\end{proof}


\subsection{Unions of subsequent block adjacent hypergraphs}
\label{sec:union-subs-block-adjacent}

Next we describe the primary decomposition of $J_\Delta$ when consecutive cliques intersect in less than $m-1$ vertices.
\begin{Definition}
\label{def:prime-sequence:2}
Let $\Delta=\Delta_1\cup\Delta_2\cup\cdots\cup\Delta_r$, where $\Delta_i$ is a block adjacent hypergraph on $[u_i,v_i]$ for each $i$, and $v_{i-1}-m+3\leq u_i$.
Then $\Delta$ is called a \emph{union of subsequent block adjacent hypergraphs}.
\end{Definition}

\begin{Theorem}
  \label{thm:block-adj-prime-split}
  Let $\Delta=\Delta_1\cup\cdots\cup\Delta_r$ be a union of subsequent block adjacent hypergraphs. Then
  $J_{\Delta_1}+\cdots+J_{\Delta_r}$ is a prime splitting of~$J_\Delta$.
\end{Theorem}
We present the proof later in Section~\ref{sec:localization-lemmata}.
Comparing with Theorem~\ref{prime adjacent}, we see that, as in Example~\ref{ex:disj-in-id-not-suff}, it is not sufficient to write an ideal as a sum of ideals with disjoint initial ideals.
In order to obtain a prime splitting, it is necessary that the block adjacent hypergraphs $\Delta_{i}$ overlap in at most $m-2$ vertices.

When $\Delta$ is a union of subsequent block adjacent hypergraphs, it follows from
Theorem~\ref{thm:block-adj-prime-split} that the minimal primes of $J_{\Delta}$ can be described in terms of prime
sequences as follows.

\begin{Definition}
Let $\Delta=\Delta_1\cup\Delta_2\cup\cdots\cup\Delta_r$ be a union of subsequent block adjacent hypergraphs.
For each $i=1,\ldots,r$ let $\Gamma_{i}$
be a prime sequence of $\Delta_i$.
Then we call $\Gamma=\{\Gamma_1,\Gamma_2,\ldots,\Gamma_r\}$ a \emph{prime sequence} of $[u_1,v_{r}]$.
Denote by $\mathcal{A}_\Delta$ the set of all prime sequences of $\Delta$.

To the prime sequence $\Gamma=\{\Gamma_1,\Gamma_2,\ldots,\Gamma_r\}$ of $\Delta$ we associate the ideal
\begin{equation*}
P_{\Gamma}=P_{\Gamma_1}+\ldots+P_{\Gamma_r} \subset K[X].
\end{equation*}
As above, if we need to restrict the row indices of the defining minors, we add a superscript: $P_{\Gamma}^{B}:=P_{\Gamma_1}^{B}+\ldots+P_{\Gamma_r}^{B}$.
\end{Definition}

\begin{Theorem}
  \label{general}
  Let $\Delta$ be a union of subsequent block adjacent hypergraphs.
  The minimal primary decomposition of $J_\Delta$ is given by
  $J_\Delta=\bigcap_{\Gamma\in \mathcal{A}_\Delta}P_{\Gamma}$.
  In particular, each ideal $P_{\Gamma}$ is prime.
\end{Theorem}

Theorem~\ref{general} is a direct consequence of Theorem~\ref{prime adjacent} and Theorem~\ref{thm:block-adj-prime-split}.
\begin{Example}
  The hypergraph from Figure~\ref{Fig:subs-block-adjacent:A} is a union $\Delta=\Delta_1\cup\Delta_2$, where
  $\Delta_1=\{123\}$ and where
  $\Delta_2=\{345\}\cup\Delta_{(3),4567}\cup\{678\}\cup\{789\}$ 
  is a block adjacent hypergraph on $\{3,4,\ldots,9\}$.  The prime sequences of $\Delta$ are
  \begin{align*}
    &\Gamma_{1}:[1,3], \Gamma_{2}:[3,9],             &
    &\Gamma_{1}:[1,3], \Gamma_{2}:[3,5],[4,9],       \\
    &\Gamma_{1}:[1,3], \Gamma_{2}:[3,5],[4,7],[7,9], &
    &\Gamma_{1}:[1,3], \Gamma_{2}:[3,5],[4,8],[7,9], \\
    &\Gamma_{1}:[1,3], \Gamma_{2}:[3,7],[7,9], &
    &\Gamma_{1}:[1,3], \Gamma_{2}:[3,7],[6,9], \\
    &\Gamma_{1}:[1,3], \Gamma_{2}:[3,8],[7,9].
  \end{align*}
\end{Example}

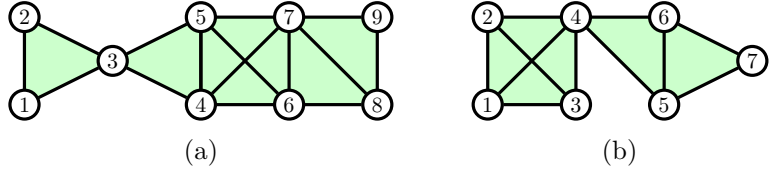
\begin{figure}[ht]
  \begin{center}
    \subcaptionbox{\label{Fig:subs-block-adjacent:A}}{\begin{tikzpicture} [scale = .58, very thick = 10mm]
      \fill[clique] (-12,-1) -- (-10,0) -- (-8,-1) -- (-4,-1) -- (-4,1) -- (-8,1) -- (-10,0) -- (-12,1);
      \node (n1) at (-12,-1)  [Cgray] {$1$};
      \node (n2) at (-12,1)  [Cgray] {$2$};
      \node (n3) at (-10,0)  [Cgray] {$3$};
      \node (n4) at (-8,-1)  [Cgray] {$4$};
      \node (n5) at (-8,1)  [Cgray] {$5$};
      \node (n6) at (-6,-1)  [Cgray] {$6$};
      \node (n7) at (-6,1)  [Cgray] {$7$};
      \node (n8) at (-4,-1)  [Cgray] {$8$};
      \node (n9) at (-4,1)  [Cgray] {$9$};
      \foreach \from/\to in {n1/n2,n1/n3,n2/n3, n3/n4,n3/n5,n4/n5, n4/n5,n4/n6,n4/n7,n5/n6,n5/n7,n6/n7, n6/n8,n7/n8, n7/n9,n8/n9}
        \draw[] (\from) -- (\to);
    \end{tikzpicture}}
  $\qquad$
  \subcaptionbox{\label{Fig:subs-block-adjacent:B}}{\begin{tikzpicture} [scale = .58, very thick = 10mm]
      \fill[clique] (0,-1) -- (2,-1) -- (2,1) -- (4,-1) -- (6,0) -- (4,1) -- (2,1) -- (0,1);
      \node (n1) at (0,-1)  [Cgray] {$1$};
      \node (n2) at (0,1)  [Cgray] {$2$};
      \node (n3) at (2,-1)  [Cgray] {$3$};
      \node (n4) at (2,1)  [Cgray] {$4$};
      \node (n5) at (4,-1)  [Cgray] {$5$};
      \node (n6) at (4,1)  [Cgray] {$6$};
      \node (n7) at (6,0)  [Cgray] {$7$};
      \foreach \from/\to in {n1/n2,n1/n3,n1/n4,n2/n3,n2/n4,n3/n4, n4/n5,n4/n6,n5/n6, n5/n7,n6/n7}
        \draw[] (\from) -- (\to);

      
    \end{tikzpicture}}
  \end{center}
  \caption{Two unions of subsequent block adjacent hypergraphs.}\label{Fig:subs-block-adjacent}
\end{figure}

\begin{Example}
  \label{ex1}
  The hypergraph from Figure~\ref{Fig:subs-block-adjacent:B} is a union $\Delta=\Delta_{(3),1234}\cup\Delta_2$, where
  $\Delta_2=\{456,567\}$ is an adjacent hypergraph on $\{4,5,6,7\}$.  The minimal primes of $J_{\Delta}$ are
  \begin{itemize}
  \item $P_1=([123],[134],[124],[234],{[12|56],[13|56],[23|56]})$,
  \item $P_2=([123],[134],[124],[234],[456],[457],[467],[567])$,
  \end{itemize}
  corresponding to the prime sequences
\[
\Gamma=\{\Gamma_1:[1,4], \Gamma_{2}:[4,6],[5,7]\}\quad\text{and}\quad
\Gamma'=\{\Gamma_1:[1,4], \Gamma_{2}:[4,7]\}.
\]
\end{Example}

\subsection{Proofs of the main results}
\label{sec:localization-lemmata}

Our strategy to prove the prime splitting is to find regular elements modulo the ideal.  Then we localize and find a
ring automorphism that transforms our ideal into a sum of ideals that are defined in disjoint sets of variables.

\begin{Lemma}[Localization lemma]
  \label{rules}
  Let $X$ be an $m\times n$-matrix of indeterminates and let $I\subset K[X]$ be an ideal generated by a set $\Gc$ of
  minors.
  Furthermore, let $i_{1},\dots,i_{k}\in[m]$ and $j_{1},\dots,j_{k}\in[n]$.  Assume that for each minor $[a_1\ldots
  a_t|b_1\ldots b_t]\in \Gc$ 
 the minors
  $[\alpha_1\ldots \alpha_t|b_1\ldots b_t]$ also belong to $\Gc$ for all
  $\{\alpha_{1},\ldots,\alpha_{t}\}\subset\{i_{1},\dots,i_{k},a_{1},\ldots,a_{t}\}$,
where $\alpha_{1}<\cdots<\alpha_{t}$.
  Then the localizations $(R/I)_{[i_{1},\dots,i_{k}|j_{1},\dots,j_{k}]}\iso
  (R/J)_{[i_{1},\dots,i_{k}|j_{1},\dots,j_{k}]}$ at the minor $[i_{1},\dots,i_{k}|j_{1},\dots,j_{k}]$ are isomorphic, where $J$ is generated by
\begin{itemize}
\item [(a)] the minors $[a_1\ldots a_t|b_1\ldots b_t]\in \Gc$ with
  $\{b_{1},\ldots,b_{t}\}\cap\{j_{1},\dots,j_{k}\} = \emptyset$,
\item
  [(b)] 
  the minors $[\alpha_{1}\ldots\alpha_{t-r}|b_1\ldots\hat{b}_{k_{1}}\ldots\hat{b}_{k_{r}}\ldots\ldots b_t]$ where
  $[a_1\ldots a_t|b_1\ldots b_t]\in \Gc$ and where $\{b_{k_{1}},\ldots,b_{k-r}\} =
  \{b_{1},\ldots,b_{t}\}\cap\{j_{1},\ldots,j_{k}\}$ and
  $\alpha_{1},\ldots,\alpha_{t-r}\in\{a_{1},\ldots,a_{t},i_{1},\ldots,i_{k}\}$.
\end{itemize}
\end{Lemma}
\begin{proof}
  For simplicity assume that $i_{1}=j_{1}=1,\dots,i_{k}=j_{k}=k$.  Let $R=K[X]$.  
  The idea of the proof is the following:
  Denote by $X_{[k]}^{[k]}$ the submatrix of $X$ that consists of the first $k$ rows and the first $k$ columns, denote by
  $X_{[k]}^{\overline{[k]}}$ the submatrix of $X$ that consists of the last $m-k$ rows 
 and the first $k$ columns, and so on.
  The matrix $X_{[k]}^{[k]}$ is invertible in the ring $R_{[1\dots k|1\dots k]}$ that arises from $R$ by localizing with respect
  to the minor $[1\dots k|1\dots k]$.  Denote the inverse by~$A$.  Thus, we can multiply the matrix $X$ from the left with the
  invertible matrix
  \begin{equation*}
    \begin{pmatrix}
      A & 0 \\
 & \\
      -X_{[k]}^{\overline{[k]}} & I_{m-k}
    \end{pmatrix},
  \end{equation*}
  where $I_{k}$ and $I_{m-k}$ denote unit matrices of corresponding sizes.  We obtain a new matrix
  \begin{equation*}
    X' :=
    \begin{pmatrix}
      A & 0 \\
 & \\
      -X_{[k]}^{\overline{[k]}}A  & I_{m-k}
    \end{pmatrix}
    \cdot X
    =
    \begin{pmatrix}
      I_{k} &  & A X_{\overline{[k]}}^{[k]} \\
 & \\
      0    &  & X_{\overline{[k]}}^{\overline{[k]}} - X_{[k]}^{\overline{[k]}} A X_{\overline{[k]}}^{[k]}
    \end{pmatrix}.
  \end{equation*}
  This transformation preserves the ranks of submatrices that arise by selecting columns, in the following sense: The
  submatrix of $X$ with columns $\{b_{1},\dots,b_{r}\}$ has rank less than $r$ if and only if the corresponding
  submatrix of $X'$ with columns $\{b_{1},\dots,b_{r}\}$ has rank less than~$r$.  If $\{b_{1},\dots,b_{r}\}$
  intersects~$[k]$, 
  then, due to the special structure of~$X'$, this rank condition is equivalent to saying that all maximal minors of
  $X'$ involving the columns $\{b_{1},\dots,b_{k}\}\setminus [k]$ vanish.  Now we take those entries $x'_{ij}$ of $X'$
  with $i,j>k$ as new variables, and we note that the transformation of variables from $\{x_{ij}:i\in[m],j\in[n]\}$ to
  $\{x_{ij}:\min\{i,j\}\le k\}\cup\{x'_{ij}:\min\{i,j\}>k\}$ is invertible, where the inverse is given by
  \begin{equation*}
    x_{ij} = x'_{ij} + \sum_{i',j'=1}^{k} x_{ij'} A_{j'i'} x_{i'j}.
  \end{equation*}
  This line of argument proves the set-theoretic variant of the statement.

  To finish the proof, it remains to show that we can express minors of $X'$ as algebraic combinations of corresponding
  minors of~$X$, and vice versa.  This follows from multilinearity of the determinant: Each row of $X'$ is a linear
  combination of rows of~$X$ with coefficients involving entries of $X$ and $[1\dots k|1\dots k]^{-1}$, and vice versa.
\end{proof}


\begin{Corollary}
\label{cor:rules}
Let $K$ be a field, $X$ be an $m\times n$-matrix of indeterminates and $I_m\subset R=K[X]$ be the ideal generated by all
$m$-minors of $X$. Furthermore, let $x_{ij}$ be an entry of $X$.  Then $(R/I_m)_{x_{ij}}\iso (R/J)_{x_{ij}}$ where $J$
is generated by the $(m-1)$-minors $[a_1\ldots a_{m-1}|b_1\ldots b_{m-1}]$ with
$i\notin\{a_{1},\dots,a_{m-1}\}$ and $j\notin\{b_{1},\dots,b_{m-1}\}$.
\end{Corollary}


The next lemma is our central tool in order to prove prime splittings.
\begin{Lemma}[Prime splitting lemma]
  \label{lem:prime-splitting}
  Assume that $K$ is algebraically closed.
  Let $\Delta=\Delta_{1}\cup\Delta_{2}$, and let $V(\Delta_{1})\cap V(\Delta_{2})=\{j_{1},\dots,j_{s}\}$.  Assume that
  there exist $s$ pairwise different elements $i_{1},\dots,i_{s}$ such that the {minor $[i_{1}\cdots
  i_{s}|j_{1}\cdots j_{s}]$} is regular modulo $J_{\Delta}$, $J_{\Delta_{1}}$ and~$J_{\Delta_{2}}$.
  Then $J_{\Delta}=J_{\Delta_{1}}+J_{\Delta_{2}}$ is a prime splitting.
\end{Lemma}
{To check that a {minor} is regular it suffices to show that none of the variables that appear in
its initial term $x_{i_{1},j_{1}}\cdots x_{i_{s},j_{s}}$ divide any of the initial terms of a Gr\"obner basis.
}
  \begin{proof}
{In the easiest case $s=1$, we just need to prove that the variable~$x_{i_{1},j_{1}}$ is regular.}
Let $y=[i_{1}\cdots i_{s}|j_{1}\cdots j_{s}]$.  Applying Lemma~\ref{rules} shows that
  $(R/J_{\Delta})_{y}\cong (R/(J_{1}+J_{2}))_{y}$, where $J_{i}$ is generated by
  \begin{enumerate}
  \item the minors $[a_{1}\dots a_{t}|b_{1}\dots b_{t}]$ with $\{b_{1},\dots,b_{t}\}\in\Delta_{i}$ that satisfy 
    $j_{\ell}\notin\{b_{1},\dots,b_{t}\}$ 
    for each~$\ell$,
  \item the minors $[a_{1}\dots a_{t'}|b_{1}\dots b_{t'}]$ whenever there exist
    $b_{t'+1},\dots,b_{t}\in\{j_{1},\dots,j_{s}\}$ with\\ $\{b_{1},\dots,b_{t}\}\in\Delta_{i}$.
  \end{enumerate}
  Thus, $J_{1}$ and $J_{2}$ are defined in disjoint sets of variables.  Since $K$ is algebraically closed, $J_{1}+J_{2}$ is a prime splitting.
  Therefore, $R_{y}J_{\Delta} = R_{y}J_{\Delta_{1}} + R_{y}J_{\Delta_{2}}$ is also a prime splitting.  Note that for any
  ideal $I$, localization with respect to~$y$ induces a bijection of those associated primes of $I$ that do not
  contain~$y$ and the associated primes of~$R_{y}I$.  The statement now follows since $y$ is regular modulo
  $J_{\Delta}$, $J_{\Delta_{1}}$ and~$J_{\Delta_{2}}$.
\end{proof}


Now we can prove Theorem~\ref{thm:block-adj-prime-split}. 
\begin{proof}[Proof of Theorem~\ref{thm:block-adj-prime-split}.] 
  We first assume that $K$ is algebraically closed.  If $u_i> v_{i-1}$ for some~$i$, then
  $\Delta$ is disconnected, and $J_{\Delta}$ can be written as a sum of two ideals defined in
  disjoint sets of variables.  Therefore, we may assume that $u_i\leq v_{i-1}$ for all~$i$.

  We do induction on~$r$.  The base case $r=1$ is Theorem~\ref{prime adjacent}.  Assume that $r>1$.  Note that vertex $v_1$
  does not take position $m-1$ in any hyperedge of~$\Delta$, and vertex $v_1-k\geq u_2$ does not take position $m-1-k$
  in any hyperedge of~$\Delta$.
  Therefore, the variables $y_1=x_{m-1,v_{1}}$, $y_2=x_{m-2,v_{1}-1}$, \dots, $y_{s}=x_{m-s,u_{2}}$, do not appear in the support of the generators of $\ini_<(P_\Gamma)$, 
  which implies that the minor {$[m-s\cdots m-1|u_{2}\cdots v_{1}]$ is} 
  regular modulo~$P_\Gamma$.  
  By Lemma~\ref{lem:prime-splitting},  $(J_{\Delta_{1}}+\dots+J_{\Delta_{r-1}})+J_{\Delta_{r}}$ is a prime splitting.  By induction, 
  $J_{\Delta_{1}}+\dots+J_{\Delta_{r-1}}$ is also a prime splitting.  So finally
  $J_{\Delta_{1}}+\dots+J_{\Delta_{r}}$ is also a prime splitting.

  If $K$ is not algebraically closed, we argue as follows: Let $L$ be the algebraic closure of~$K$.  Then, using the
  notation of Lemma~\ref{lem:field-lemma}, Theorem~\ref{thm:block-adj-prime-split} holds for $(J_{\Delta})_{L}$, and the
  primary decomposition of $(J_{\Delta})_{L}$ follows from Theorem~\ref{general}.  Since the primary components of
  $(J_{\Delta})_{L}$ are all defined over~$K$, the primary decomposition of $J$ has the same structure.  It follows that
  $J$ is a prime splitting.
\end{proof}

{The next example illustrates the idea of the last proof.} 

\begin{Example}
  Let $m=4$ and $\Delta=\Delta_1\cup \Delta_2\cup\Delta_3$ for the block adjacent hypergraphs
  $\Delta_1=\Delta_{(4),12345}\cup\{3456\}$, $\Delta_2=\{5678\}$ and $\Delta_3=\{\{7,8,9,10\}\}$. The variables $x_{25}$ and $x_{36}$ do not appear in any initial term (with respect to the lexicographic order) and the minor $[23|56]=x_{25}x_{36}-x_{26}x_{35}$ is regular  modula $J_\Delta$. Also, the variables $x_{27}$ and $x_{38}$ do not appear in any initial term and also the minor $[23|78]=x_{27}x_{38}-x_{28}x_{37}$ is regular modulo $J_\Delta$. Therefore, $J_\Delta=J_{\Delta_1}+J_{\Delta_2}+J_{\Delta_3}$ is a prime splitting. More precisely, 
the prime sequences of $\Delta$ are
$$\Gamma = \{ \Gamma_{1}:[1,5],[3,6],\ \Gamma_{2}:[5,8], \ \Gamma_{3}:[7,10] \}\quad \text{ and}\quad \Gamma' = \{\Gamma_{1}:[1,6], \  \Gamma_{2}:[5,8], \ \Gamma_{3}:[7,10]\}.$$
  \end{Example}

\section{Further applications of the prime splitting lemma}
\label{sec:further-examples}

In Theorem~\ref{thm:glue-along-cliques} we have seen an example of how hypergraphs can be glued along cliques
to construct larger hypergraphs which still share some of the properties of the smaller hypergraphs.  In this
section, we look at other instances of prime splittings.  Our main tool is the prime splitting lemma (Lemma~\ref{lem:prime-splitting}).
The following theorem is a simple reformulation of this lemma.
\begin{Theorem}
  \label{thm:lemma-as-thm}
  Suppose that $\Delta$ is a closed hypergraph that can be written as a union
  $\Delta_1\cup\Delta_2\cup\cdots\cup\Delta_r$ of closed sub-hypergraphs. 
  For $k=1,\dots,r$ suppose that $V(\Delta_{1}\cup\dots\cup\Delta_{k-1})\cap
  V(\Delta_{k})=\{j_{k,1},\dots,j_{k,r_{k}}\}$, and suppose that there exist pairwise different indices
  $i_{k,1},\dots,i_{k,r_{k}}$ such that $x_{i_{k,\ell},j_{k,\ell}}$ does not appear in any initial term of the
  generators of~$J_{\Delta}$.  Then $J_{\Delta}=J_{\Delta_{1}}+\dots+J_{\Delta_{r}}$ is a prime splitting.
\end{Theorem}
\begin{proof}
  The statement follows inductively from Lemma~\ref{lem:prime-splitting}.  Note that the assumptions imply that
  $[i_{k,1}\cdots i_{k,r_{k}}|j_{k,1}\cdots j_{k,r_{k}}]$
  is regular modulo all relevant ideals.
\end{proof}
\begin{Remark}
  In the context of Theorem~\ref{thm:lemma-as-thm}:
  \begin{enumerate}
  \item \color{black}
    If each $\Delta_{i}$ is itself a union of subsequent block adjacent hypergraphs, then the minimal primes can be
    described in terms of prime sequences of the~$\Delta_{i}$.
  \item
    If each $\Delta_{i}$ is a clique, then $J_{\Delta}$ is prime.
  \end{enumerate}
\end{Remark}
\begin{Example}
  \label{ex2}
  Let $m=4$ and $\Delta=\Delta_1\cup\Delta_2\cup\Delta_3$ be the $3$-dimensional hypergraph on the vertex set $[11]$, where $\Delta_1=\{1234,2345\}$, $\Delta_2=\{4678,6789\}$ and $\Delta_3=\{\{5,9,10,11\}\}$, see Figure~\ref{Fig:minimal-primes}.
  By Theorem~\ref{general}, the associated primes of $J_{\Delta_1}$ are $J_{\Delta_{(4),12345}}$ and $J_{\{234\}}$.
  The prime ideals of $J_{\Delta_2}$ are $J_{\Delta_{(4),46789}}$ and~$J_{\{678\}}$.
  One can see that $x_{2,4}$ does not appear in any of the initial terms of the defining minors of $J_{\Delta_{1}}$,
  $J_{\Delta_{2}}$ and~$J_{\Delta_{1}}+J_{\Delta_{2}}$, which form $<_{\lex}$-Gr\"obner basis of the corresponding
  ideals.  Thus, $x_{2,4}$ is regular modulo $J_{\Delta_{1}}$, $J_{\Delta_{2}}$ and~$J_{\Delta_{1}}+J_{\Delta_{2}}$, and
  $J_{\Delta_{1}}+J_{\Delta_{2}}$ is a prime splitting.  Similarly, {$[13|59]=x_{15}x_{39}-x_{19}x_{35}$ is a regular element}
  modula $J_{\Delta_{1}}+J_{\Delta_{2}}$, $J_{\Delta_{3}}$ and~$J_{\Delta}$, and so $(J_{\Delta_{1}}+J_{\Delta_{2}}) +
  J_{\Delta_{3}}$ is a prime splitting.  In total, $J_{\Delta}=J_{\Delta_{1}}+J_{\Delta_{2}}+J_{\Delta_{3}}$ is a prime
  splitting.  The minimal prime ideals of $J_\Delta$ are
  \begin{gather*}
    J_{\Delta_{(4),12345}} + J_{\Delta_{(4),46789}} + J_{\Delta_{3}}, \qquad
    J_{\Delta_{(4),12345}} + J_{\{678\}} + J_{\Delta_{3}}, \\
    J_{\{234\}} + J_{\Delta_{(4),46789}} + J_{\Delta_{3}}, \qquad
    J_{\{234\}} + J_{\{678\}} + J_{\Delta_{3}}.
\end{gather*}
\end{Example}

\begin{figure}[ht]
  \begin{center}
    \begin{tikzpicture} [scale = .58, very thick = 10mm]
      \filldraw[clique1] (1.7786406,3.3359217) arc[start angle=251.56505,end angle=90,radius=.7]
             -- (4,4.7) arc[start angle=90,end angle=63.43,radius=.7]
             -- (8.3130495,2.626099) arc[start angle=63.43,end angle=-108.43495,radius=.7]
             -- cycle;
      \filldraw[clique2] (-0.15811389,3.506803) arc[start angle=260.53768,end angle=90,radius=.5]
             -- (4,4.5) arc[start angle=90,end angle=63.43,radius=.5]
             -- (6.2236068,3.4472136) arc[start angle=63.43,end angle=-99.462322,radius=.5]
             -- cycle;
      \filldraw[clique3] (-0.15811389,2.493197) arc[start angle=-260.53768,end angle=-90,radius=.5]
             -- (4,1.5) arc[start angle=-90,end angle=-63.43,radius=.5]
             -- (6.2236068,2.5527864) arc[start angle=-63.43,end angle=+99.462322,radius=.5]
             -- cycle;
      \filldraw[clique4] (-0.15811389,1.506803) arc[start angle=260.53768,end angle=90,radius=.5]
             -- (4,2.5) arc[start angle=90,end angle=63.43,radius=.5]
             -- (6.2236068,1.4472136) arc[start angle=63.43,end angle=-99.462322,radius=.5]
             -- cycle;
      \filldraw[clique5] (12.158114,1.506803) arc[start angle=-80.53768,end angle=90,radius=.5]
             -- (8,2.5) arc[start angle=90,end angle=116.57,radius=.5]
             -- (5.7763932,1.4472136) arc[start angle=116.57,end angle=279.46232,radius=.5]
             -- cycle;
      \node (n1) at (0,4)  [Cgray] {$1$};
      \node (n2) at (2,4)  [Cgray] {$2$};
      \node (n3) at (4,4)  [Cgray] {$3$};
      \node (n4) at (6,3)  [Cgray] {$4$};
      \node (n5) at (8,2)  [Cgray] {$5$};
      \node (n6) at (0,2)  [Cgray] {$6$};
      \node (n7) at (2,2)  [Cgray] {$7$};
      \node (n8) at (4,2)  [Cgray] {$8$};
      \node (n9) at (6,1)  [Cgray] {$9$};
      \node (n10) at (10,2) [Cgray2] {$10$};
      \node (n11) at (12,2) [Cgray2] {$11$};
    \end{tikzpicture}
    \caption{A $3$-dimensional hypergraph.}\label{Fig:minimal-primes}
  \end{center}
\end{figure}
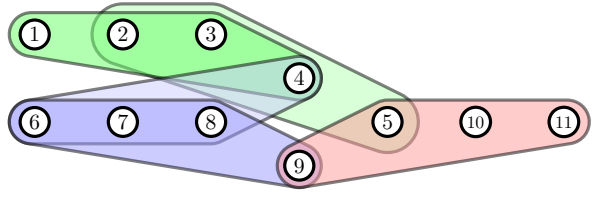
\vspace{-.5cm}
\begin{Example}
  Let $m=3$, and let $\Delta=\Delta_1\cup\Delta_2\cup\Delta_3\cup \Delta_4$, where $\Delta_1$ is the adjacent hypergraph
  on the vertex set $\{1,2,3,4\}$, $\Delta_2$ is the adjacent hypergraph on $\{4,5,6,7\}$, $\Delta_3$ is the adjacent
  hypergraph on $\{3,7,8,9\}$ and $\Delta_4 = \Delta_{(3),\{9,10,11,12\}}$, 
see Figure~\ref{fig:cactus}.
It is straightforward to check that $\Delta$ is closed.
As we know $J_{\Delta_4}$ is a prime ideal, and by Theorem~\ref{prime adjacent},
the minimal primary decompositions of $J_{\Delta_1}$, $J_{\Delta_2}$ and $J_{\Delta_3}$ are
\begin{itemize}
\item $J_{\Delta_1}=J_{\Delta_{1,1}}\cap J_{\Delta_{1,2}}$, where $\Delta_{1,1}=\{23\}$ and $\Delta_{1,1}=\{1234\}$.
\item $J_{\Delta_2}=J_{\Delta_{2,1}}\cap J_{\Delta_{2,2}}$, where $\Delta_{2,1}=\{56\}$ and $\Delta_{2,1}=\{4567\}$.
\item $J_{\Delta_3}=J_{\Delta_{3,1}}\cap J_{\Delta_{3,2}}$, where $\Delta_{2,1}=\{78\}$ and $\Delta_{2,1}=\{3789\}$.
\end{itemize}
Then the minimal primes of $J_\Delta$ are 
\begin{equation*}
  J_{\Delta_4}+J_{\Delta_{1,i}}+J_{\Delta_{2,j}}+J_{\Delta_{3,k}}\quad \text{ for }\ i=1,2,\ j=1,2,\ k=1,2.
\end{equation*}
\end{Example}

\begin{figure}[ht]
  \begin{center}
    \begin{tikzpicture} [scale = .70, very thick = 10mm] 
      \fill[clique] (-11,0) -- (-12,-1) -- (-11,-2) -- (-10,-1) -- (-11,0) -- (-11,1) -- (-9,1) -- (-7,0) -- (-6,1) -- (-5,0) -- (-6,-1) -- (-7,0) -- (-8,-1) -- (-10,-1) -- (-9,1);
      \node (n1) at (-11,-2)  [Cgray] {$1$};
      \node (n2) at (-12,-1)  [Cgray] {$2$};
      \node (n3) at (-10,-1)  [Cgray] {$3$};
      \node (n4) at (-11,0)  [Cgray] {$4$};
      \node (n5) at (-11,1)  [Cgray] {$5$};
      \node (n6) at (-10,0.5)  [Cgray] {$6$};
      \node (n7) at (-9,1)  [Cgray] {$7$};
      \node (n8) at (-8,-1)  [Cgray] {$8$};
      \node (n9) at (-7,0)  [Cgray] {$9$};
      \node (n10) at (-6,1)  [Cgray2] {$10$};
      \node (n11) at (-6,-1)  [Cgray2] {$11$};
      \node (n12) at (-5,0)  [Cgray2] {$12$};
      \foreach \from/\to in {n1/n2,n1/n3,n2/n3,n2/n4,n3/n4, n4/n5,n4/n6,n5/n6,n5/n7,n6/n7, n3/n7,n3/n8,n7/n8,n7/n9,n8/n9, n9/n10,n9/n11,n9/n12,n10/n11,n10/n12,n11/n12}
        \draw[] (\from) -- (\to);
   \end{tikzpicture}
  \end{center}
  \caption{A block adjacent hypergraph. 
  }
  \label{fig:cactus}
\end{figure}
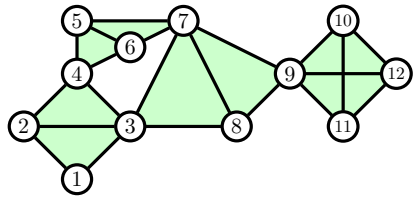

\begin{Example}
\label{ex3}
Assume that $\Delta=\Delta_1\cup\Delta_2\cup\Delta_3$ is the $3$-dimensional hypergraph on $[9]$, where $\Delta_1=\{1234\}$, $\Delta_2=\Delta_{(3),2567}$, 
and $\Delta_3=\{489\}$, see Figure~\ref{Fig:tree-of-cliques}.
The hypergraph $\Delta$ satisfies the conditions of Theorem~\ref{thm:lemma-as-thm}. 
  Hence, $J_\Delta$ is a prime ideal.
\end{Example}

\begin{figure}[ht]
  \begin{center}
    \begin{tikzpicture} [scale = .58, very thick = 10mm]
      \filldraw[clique2] (3.7763932,2.4472136) arc[start angle=116.56505,end angle=63.434949,radius=.5]
             -- (6.2236068,1.4472136) arc[start angle=63.434949,end angle=-63.434949,radius=.5]
             -- (4.2236068,-0.4472136) arc[start angle=296.56505,end angle=243.43495,radius=.5]
             -- (1.7763932,0.5527864) arc[start angle=243.43495,end angle=116.56505,radius=.5]
             -- cycle;
      \filldraw[clique4] (-0.5,2) arc[start angle=180,end angle=63.434949,radius=.5]
             -- (2.2236068,1.4472136) arc[start angle=63.434949,end angle=-63.434949,radius=.5]
             -- (0.2236068,-0.4472136) arc[start angle=296.56505,end angle=180,radius=.5]
             -- cycle;
      \filldraw[clique5] (7.3,2) arc[start angle=180,end angle=63.434949,radius=.7]
             -- (10.3130495,1.626099) arc[start angle=63.434949,end angle=-63.434949,radius=.7]
             -- (8.3130495,-0.62609903) arc[start angle=296.56505,end angle=180,radius=.7]
             -- cycle;
      \filldraw[clique5] (8.7,2) arc[start angle=0,end angle=116.56505,radius=.7]
             -- (5.6869505,1.626099) arc[start angle=116.56505,end angle=243.43495,radius=.7]
             -- (7.6869505,-0.62609903) arc[start angle=-116.56505,end angle=0,radius=.7]
             -- cycle;
      \draw[semitransparent] (7.7763932,2.4472136) arc[start angle=116.56505,end angle=63.434949,radius=.5]
             -- (10.2236068,1.4472136) arc[start angle=63.434949,end angle=-90,radius=.5]
             -- (6,0.5) arc[start angle=270,end angle=116.56505,radius=.5]
             -- cycle;
      \draw[semitransparent] (7.7316718,-0.53665632) arc[start angle=-116.56505,end angle=-63.434949,radius=.6]
             -- (10.2683281,0.46334368) arc[start angle=-63.434949,end angle=90,radius=.6]
             -- (6,1.6) arc[start angle=-270,end angle=-116.56505,radius=.6]
             -- cycle;
      \node (n1) at (4,2)  [Cgray] {$1$};
      \node (n3) at (6,1)  [Cgray] {$3$};
      \node (n2) at (4,0)  [Cgray] {$2$};
      \node (n4) at (2,1)  [Cgray] {$4$};
      \node (n5) at (8,2)  [Cgray] {$5$};
      \node (n6) at (8,0)  [Cgray] {$6$};
      \node (n7) at (10,1) [Cgray] {$7$};
      \node (n8) at (0,2)  [Cgray] {$8$};
      \node (n9) at (0,0)  [Cgray] {$9$};
    \end{tikzpicture}
    \caption{A $3$-dimensional hypergraph. 
    }\label{Fig:tree-of-cliques}
  \end{center}
\end{figure}
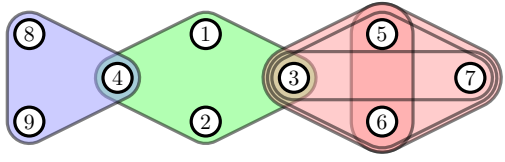

{One may ask how tight is Theorem~\ref{thm:lemma-as-thm}.  Are all conditions in the statement necessary?  The example of the block adjacent hypergraphs shows that if cliques have a large overlap, then there is no prime splitting among them.  Also, the ideal of a union of subsequent block adjacent hypergraphs is prime if and only if all assumptions of Theorem~\ref{thm:lemma-as-thm} are satisfied.}
{That the closedness assumption is necessary can be seen already from the case of binomial edge ideals~\cite{Rauh13,HHH}.
  For illustrative purposes, we present a higher-dimensional example.}
\begin{Example}
  \label{ex33}
  Assume that $\Delta=\Delta_1\cup\Delta_2\cup\Delta_3\cup\Delta_4$ is the $2$-dimensional hypergraph on $[6]$, where
  $\Delta_1=\{123\}$, $\Delta_2=\{345\},\ \Delta_{3}=\{146\}$, and $\Delta_4=\{256\}$, see Figure~\ref{fig:beautiful}.
  However, one may check with {Macaulay2}~\cite{M2} that $J_\Delta$ is not a prime ideal (the computation is not straight-forward, though, and we are grateful to Michael Stillman for helping us with this example).
\end{Example}

\begin{figure}[ht]
  \begin{center}
    \begin{tikzpicture} [scale = .70, very thick = 10mm] 
      \fill[clique] (-1,0) -- (1,0) -- (0,1.5)
                    (0,1.5) -- (-3,2) -- (3,2)
                    (-1,0) -- (-3,2) -- (0,-2)
                    (1,0) -- (3,2) -- (0,-2)
                    ;
      \node (n1) at (-1,0)    [Cgray] {$1$};
      \node (n2) at ( 1,0)    [Cgray] {$2$};
      \node (n3) at ( 0,1.5)  [Cgray] {$3$};
      \node (n4) at (-3,2)    [Cgray] {$4$};
      \node (n5) at ( 3,2)    [Cgray] {$5$};
      \node (n6) at (0,-2)  [Cgray] {$6$};
      \draw (n1) edge (n2) edge (n3) edge (n4) edge (n6)
            (n2) edge (n3) edge (n5) edge (n6)
            (n3) edge (n4) edge (n5)
            (n4) edge (n5) edge (n6)
            (n5) edge (n6);
    \end{tikzpicture}
  \end{center}
  \caption{A hypergraph that is not closed. 
  }
  \label{fig:beautiful}
\end{figure}
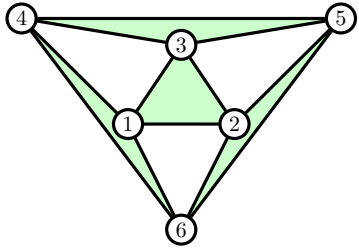

\section{Minimal free resolution of determinantal ideals}
\label{sec:res}

In this part we study the minimal free resolution of the determinantal ideal for a closed hypergraph $\Delta$ with clique decomposition $\Delta=\Delta_1\cup\cdots\cup \Delta_r$. We construct the minimal free resolution of $J_\Delta$ as a tensor product of the minimal free resolution of the determinantal ideals of its cliques. 
Here we state the following known result from~\cite{Eisenbud}.

\begin{Proposition}
\label{clique}
Recall that $I_{(m)}$ denotes the ideal generated by all maximal minors of~$X$. The minimal free resolution of $R/I_{(m)}$ is given by the Eagon-Northcott complex
\begin{multline*}
  \dots \xrightarrow{\partial_{k+1}} D_{k}(S^{m})\otimes\wedge^{k+m}(S^{n})
        \xrightarrow{\partial_{k}} D_{k-1}(S^m)\otimes \wedge^{k+m-1}(S^n)
        \xrightarrow{\partial_{k-1}} \dots \\
  \dots \xrightarrow{\partial_{2}} D_{1}(S^{m})\otimes \wedge^{m+1}(S^{n}) \xrightarrow{\partial_{1}} 
  \wedge^{m}(S^{n}) \xrightarrow{\partial_{0}} R \xrightarrow{\partial_{0}} R/I_{(m)},
\end{multline*}
where the first map $\partial_{1}$ consists of the elements of $G(I_{(m)})$, i.e., the generating set of $I_{(m)}$.
Here $D_k$ is the divided power algebra, and matrices are chosen with respect to natural basis elements $e_{i_1}\ldots e_{i_k}\otimes g_{j_1}\wedge g_{j_2}\wedge \cdots\wedge g_{j_{k+m}}$ for $i_1\leq i_2\leq\cdots\leq i_k$ and $j_1<\cdots<j_{k+m}$, where $e_1,\ldots,e_m$ are denoted for  basis elements of rows of $X$, and $g_1,\ldots,g_n$ for  basis elements of columns of $X$.
\end{Proposition}

\begin{Remark}\label{eagon}
The Eagon-Northcott complex associated to $R/I_{(m)}$ is a minimal \emph{linear} free resolution of $R/I_{(m)}$, and the Betti numbers are given by
\[
\beta_{i}(I_{(m)})=\beta_{i}(\ini_<I_{(m)})=\binom{n}{m+i}\binom{m+i-1}{i}\quad {\rm for\ all}\ i,
\]  for any term order $<$, (see~\cite{BZ} for details).
\end{Remark}

The tensor product $(A\otimes B,\partial)$ of two chain complexes $(A,d_1)$ and $(B,d_2)$ is defined by $(A \otimes B)_k = \bigoplus_{i+j=k} A_i\otimes B_j$ and $\partial(a\otimes b) = d_1a \otimes b + (-1)^i a \otimes d_2b$, when $a\in A_i$. Then $\partial^2 = 0$, and $\partial$ induces a natural map $\partial:H(A)\otimes H(B) \rightarrow H(A\otimes B)$ such that $\partial(a\otimes b)=a\otimes b$. If $a=d_1c$ is a boundary and $b$ is a cycle, then $a \otimes b = \partial(c\otimes b)$ is again a  boundary which shows that
$\partial$ is well-defined.
Let $R=K[x_1,\ldots,x_n]$ and $a = (a_1,\ldots,a_n)$ be a vector, where all $a_\ell$'s are positive integers. Then the \emph{$a$-grading} is the graded
structure induced by $a$ on $R$ which considers $a_\ell$ as the degree of $x_\ell$ for all $\ell$.
The \emph{$a$-degree} of the monomial $m=x_1^{d_1}\cdots x_{n}^{d_n}$ is $a_1d_1+\cdots+a_nd_n$, and the {$a$-degree} of a polynomial $f=\sum_{i=1}^r\lambda_i m_i$ denoted by $a(f)$, is the largest $a$-degree of a monomial in $f$. Then $\ini_a(f):=\lambda_{j_1}m_{j_1}+\cdots+\lambda_{j_k}m_{j_k}$, where $a(f)=a(m_{j_1})=\cdots=a(m_{j_k})$, and $a(f)>a(m_i)$ for $m_i\neq m_{j_\ell}$.

The following lemma is a consequence of the K\"unneth formula for the exactness of the tensor product of two exact complexes of $K$-vector spaces. To make the paper self-contained we include a short proof here.
\begin{Lemma}\label{initial}
Let $I=I_1+I_2+\cdots+I_r$ for ideals $I_\ell\subset K[X_\ell]$, where $X_\ell\subset\{x_1,\ldots,x_n\}$ and $X_\ell\cap X_{\ell'}=\emptyset$ for all $\ell<\ell'$. Assume that $\mathcal{F}_{i}$ is the minimal free resolution of $R/I_i$ for each $i$.
Then the minimal free resolution of $R/I$ is obtained by $\mathcal{F}_{1}\otimes \mathcal{F}_{2}\otimes \cdots\otimes \mathcal{F}_{r}$.
\end{Lemma}
\begin{proof}
The proof is by induction on $r$. Assume that $r>1$. Since differential maps of the tensor complex are defined in terms of differential maps of $\mathcal{F}_\ell$'s, the minimality of the tensor complex follows by the minimality of the resolutions of all components.
On the other hand, these ideals live in rings with disjoint variables which implies that $\Tor_i(R/(I_1+\cdots+I_{r-1}),R/I_r)=0$ for $i>0$, and so the constructed complex is indeed a minimal free resolution for $R/I$.
\end{proof}

\begin{Theorem}\label{res}
Let $\Delta=\Delta_1\union\Delta_2\union\cdots\union \Delta_r$ be the clique decomposition of a closed hypergraph $\Delta$. Assume that $\mathcal{F}_{i}$ is the minimal free resolution of $R/J_{\Delta_i}$ given by the Eagon-Northcott complex for each $i$.
Then the minimal free resolution of $R/J_\Delta$ is obtained by $\mathcal{F}_{1}\otimes \mathcal{F}_{2}\otimes \cdots\otimes \mathcal{F}_{r}$.
\end{Theorem}
\begin{proof}
Assume that $n_\ell=|V(\Delta_\ell)|$ and $t_\ell=\dim(\Delta_\ell)$ for each $\ell$.  Note that in the case $n_\ell= t_\ell+1$, the ideal $J_{\Delta_{\ell}}$ can be identified with the determinantal ideal of the $n_\ell$-clique on $t_\ell+1$ vertices.
Therefore we can get a minimal free resolution of $J_{\Delta}$ by Eagon-Northcott complex. 
The proof is by induction on $r$. Let $r>1$, $I=J_{\Delta_1}+\cdots+J_{\Delta_{r-1}}$ and $L=J_{\Delta_r}$. By induction hypothesis assume that ${F_\cdot}=\mathcal{F}_{1}\otimes \mathcal{F}_{2}\otimes \cdots\otimes \mathcal{F}_{r-1}$ is a minimal free resolution of $R/I$, and $G_\cdot$ is a minimal free resolution of $L$.
We consider a weight vector $a=(a_{11},a_{12},\ldots,a_{mn})$ in $N^{mn}$ such that
$$
{\ini_a([c_1\cdots c_t|i_1\cdots i_t])=\ini_{<_{\lex}}([c_1\cdots c_t|i_1\cdots i_t]),\ \ini_a(I)=\ini_{<_{\lex}}(I)\ {\rm and}\ \ini_a(L)=\ini_{<_{\lex}}(L).}
$$
Assume that $b=(1,1,\ldots,1)\in N^{mn}$. Then both $I$ and $L$ are $b$-homogenous. For all $u\in\ini_{<_{\lex}}(I)$ and $v\in\ini_{<_{\lex}}(L)$, we have $\supp(u)\cap\supp(v)=\emptyset$ and so the generators of the ideals $\ini_{<_{\lex}}(I)$ and $\ini_{<_{\lex}}(L)$ are relatively prime which implies $\Tor_i(R/\ini_{<_{\lex}}(I),R/\ini_{<_{\lex}}(L))=0$ for $i>0$. Now applying~\cite[Proposition~3.3]{BrunsConca} we conclude that
$\Tor_i(R/I,R/L)=0$ for $i>0$
which implies $F_\cdot\otimes G_\cdot$ is a resolution for $R/J_\Delta$. Since differential maps are defined in terms of differential maps in $F_\cdot$ and $G_\cdot$, we have $\partial_i(F_\cdot\otimes G_\cdot)_i\subset \mm(F_\cdot\otimes G_\cdot)_{i-1}$ which is equivalent to the minimality of the resolution.
\end{proof}

\begin{Example}
The ideal 
of the hypergraph in Figure~\ref{Fig:subs-block-adjacent:B}
is the sum of three maximal determinantal ideals.  The resolution of each component is given by the Eagon-Northcott complex as follows:
\begin{eqnarray*}
0
 \longrightarrow R(-4)^{3}
 \xrightarrow{
A=\left(\begin{array}{ccc}
  z_{1} &  - y_{1} & x_{1} \\
   - z_{2} & y_{2} &  - x_{2} \\
  z_{3} &  - y_{3} & x_{3} \\
   - z_{4} & y_{4} &  - x_{4} \end{array}\right) } R(-3)^{4}
 &\xrightarrow{([234]\ \ [134]\ \ [124]\ \ [123])} R
&\\
0
 \longrightarrow R(-3)
 &\xrightarrow{\ \quad\quad\quad([456])\quad\quad\quad} R
&\\
0
 \longrightarrow R(-3)
 &\xrightarrow{\ \quad\quad\quad([567])\quad\quad\quad} R
\end{eqnarray*}
The tensor complex of these resolutions is:
\[
0
 \longrightarrow R(-10)^{3}
 \xrightarrow{d_3} R(-7)^{6} \oplus R(-9)^{4}
 \xrightarrow{d_2} R(-4)^{3} \oplus R(-6)^{9}
\xrightarrow{d_1} R(-3)^{6}
  \xrightarrow{([567]\ [456] \ \cdots \ [123]) } R
\]
For example, basis elements of the last module $R(-10)^{3}$ in the
resolution are:
\[
[1123|1234],\
[1223|1234],\
[1233|1234],
\]
where by $[ijk\ell|1234]$ we mean the determinant of the submatrix with row indices $i,j,k,\ell$ (not necessarily distinct) and column indices $1,2,3,4$.
The differential map acts on the basis element $[1233|1234]\otimes[456]\otimes[567]$ as:
\begin{eqnarray*}
d_3([1233|1234]\otimes[456]\otimes[567])&=&\partial_{1,1}([1233|1234])+(-1)^2 \partial_{2,0}([456])+(-1)^3\partial_{3,0}([567])
\\&=&
(z_1[234]-z_2[134]+z_3[124]-z_4[123])\otimes[456]\otimes [567]\\&+&[456]([1233|1234]\otimes1\otimes[567])\\&-&[567]([1233|1234]\otimes[456]\otimes1),
\end{eqnarray*}
where $\partial_{i,j}$ is the $j^{\rm th}$ differential map in the resolution of the $i^{\rm th}$ ideal.
\end{Example}

Remark~\ref{eagon} and Lemma~\ref{initial} together with Theorem~\ref{res} imply the following result.
\begin{Corollary}\label{Betti}
Let $\Delta$ be a closed hypergraph. Then
\begin{itemize} 
\item[(a)] $\beta_{i,j}(\ini_<(J_{\Delta}))=\beta_{i,j}(J_{\Delta})$ for all $i,j.$

\item[(b)] $J_\Delta$ has a minimal {\rm linear} resolution if and only if $\Delta$ is a clique. 
\end{itemize}
\end{Corollary}

{We remark that the arguments leading to Corollary~\ref{Betti} are similar to those used in the proof
of~\cite[Proposition~3.2]{EneHerzogHibi11:CM_binomial_edge_ideals}.  Our Corollary extends ~\cite[Proposition~3.2]{EneHerzogHibi11:CM_binomial_edge_ideals} from graphs to hypergraphs. }

\begin{Remark}
Determinantal ideals associated to closed hypergraphs are examples of ideals with ``nice'' initial ideals in the sense of~\cite{ConcaHostenThomas, Adam, FatemehFarbod}, since the Betti numbers of these ideals are equal to the Betti numbers of their initial ideals. In other words, Corollary~\ref{Betti} gives a positive answer to~\cite[Question~1.1]{ConcaHostenThomas}.
\end{Remark}

\section{Further questions}

So far our analysis is restricted to the case of closed hypergraphs, where a Gr\"obner basis is known.  In this case, we
can find variables that are regular and use the localization lemma.  Thus, to generalize our results to
non-closed hypergraphs, it is necessary to either understand Gr\"obner bases for more general hypergraphs, or to have
methods to find regular elements without knowing a Gr\"obner basis.

Another direction in which to generalize the results is as follows: given a pair of 
hypergraphs $(\Delta_{1},\Delta_{2})$, 
we can associate a determinantal ideal
\[
J_{\Delta_1,\Delta_2}=([a_1\ldots a_k| b_1\ldots b_k]\:\; \{a_1,\ldots,a_k\}\in \Delta_1 \text{ and } \{b_1,\ldots,b_k\}\in \Delta_2).
\]
In our computations using the software Singular~\cite{GPS} we observed that $J_{\Delta_1,\Delta_2}$ is a radical ideal only in the case that either $\Delta_1$ or $\Delta_2$ is a clique.  We are interested to see how algebraic properties of $J_{\Delta_1}$ and $J_{\Delta_2}$ influence algebraic properties of $J_{\Delta_1,\Delta_2}$.
The case in which both $\Delta_1$ and $\Delta_2$ are graphs was studied in~\cite{Ene}.





\bigskip
\textbf{Acknowledgements.}
We would like to thank H${\rm \acute{e}}$l${\rm \grave{e}}$ne Barcelo, J\"urgen Herzog and Volkmar Welker for many suggestions and helpful discussions. We are grateful to Michael Stillman for helping us with Example~\ref{ex33}.
The first author acknowledges support
from the Mathematical Sciences Research Institute (MSRI), and the
Alexander von Humboldt Foundation during this project.

\bibliography{Det2015}

\begin{thebibliography}{10}
\providecommand{\url}[1]{#1}
\csname url@samestyle\endcsname
\providecommand{\newblock}{\relax}
\providecommand{\bibinfo}[2]{#2}
\providecommand{\BIBentrySTDinterwordspacing}{\spaceskip=0pt\relax}
\providecommand{\BIBentryALTinterwordstretchfactor}{4}
\providecommand{\BIBentryALTinterwordspacing}{\spaceskip=\fontdimen2\font plus
\BIBentryALTinterwordstretchfactor\fontdimen3\font minus
  \fontdimen4\font\relax}
\providecommand{\BIBforeignlanguage}[2]{{%
\expandafter\ifx\csname l@#1\endcsname\relax
\typeout{** WARNING: IEEEtranS.bst: No hyphenation pattern has been}%
\typeout{** loaded for the language `#1'. Using the pattern for}%
\typeout{** the default language instead.}%
\else
\language=\csname l@#1\endcsname
\fi
#2}}
\providecommand{\BIBdecl}{\relax}
\BIBdecl

\bibitem{BZ}
D.~{Bernstein} and A.~{Zelevinsky},
  ``\BIBforeignlanguage{English}{{Combinatorics of maximal minors}},''
  \emph{\BIBforeignlanguage{English}{{J. Algebr. Comb.}}}, vol.~2, no.~2, pp.
  111--121, 1993.

\bibitem{Adam}
A.~Boocher, ``Free resolutions and sparse determinantal ideals,'' \emph{Math.
  Res. Lett.}, vol.~19, no.~4, pp. 805--821, 2012.

\bibitem{BC}
W.~{Bruns} and A.~{Conca}, ``\BIBforeignlanguage{English}{{G}r\"obner bases and
  determinantal ideals},'' in \emph{\BIBforeignlanguage{English}{{Commutative
  algebra, singularities and computer algebra. Proceedings of the NATO advanced
  research workshop, Sinaia, Romania, September 17--22, 2002}}}.\hskip 1em plus
  0.5em minus 0.4em\relax Dordrecht: Kluwer Academic Publishers, 2003, pp.
  9--66.

\bibitem{BrunsConca}
W.~Bruns and A.~Conca, ``{G}r{\"o}bner bases, initial ideals and initial
  algebras,'' 2003, arXiv preprint math.AC/0308102.

\bibitem{BH1}
W.~Bruns and J.~Herzog, ``On the computation of a-invariants,''
  \emph{Manuscripta Math.}, vol.~77, no.~1, pp. 201--213, 1992.

\bibitem{BV}
W.~Bruns and U.~Vetter, ``Determinantal rings, volume 1327 of lecture notes in
  mathematics,'' 1988.

\bibitem{CGG}
L.~Caniglia, J.~A. Guccione, and J.~J. Guccione, ``Ideals of generic minors,''
  \emph{Comm. Algebra}, vol.~18, no.~8, pp. 2633--2640, 1990.

\bibitem{ConcaHostenThomas}
\BIBentryALTinterwordspacing
A.~Conca, S.~Ho{\c{s}}ten, and R.~R. Thomas, ``Nice initial complexes of some
  classical ideals,'' in \emph{Algebraic and geometric combinatorics}, ser.
  Contemp. Math.\hskip 1em plus 0.5em minus 0.4em\relax Providence, RI: Amer.
  Math. Soc., 2006, vol. 423, pp. 11--42. [Online]. Available:
  \url{http://dx.doi.org/10.1090/conm/423/08073}
\BIBentrySTDinterwordspacing

\bibitem{A}
A.~Conca, ``Ladder determinantal rings,'' \emph{Pure Appl. Algebra}, vol.~98,
  pp. 119--134, 1995.

\bibitem{Aldo}
------, ``{G}orenstein ladder determinantal rings,'' \emph{J. Lond. Math.
  Soc.}, vol.~54, no.~3, pp. 453--474, 1996.

\bibitem{CH}
A.~Conca and J.~Herzog, ``On the {H}ilbert function of determinantal rings and
  their canonical module,'' \emph{Proc. Amer. Math. Soc.}, pp. 677--681, 1994.

\bibitem{GPS}
W.~Decker, G.-M. Greuel, G.~Pfister, and H.~Sch\"onemann, ``{\sc Singular}
  {4-0-2} --- {A} computer algebra system for polynomial computations,''
  \url{http://www.singular.uni-kl.de}, 2015.

\bibitem{Eisenbud}
D.~Eisenbud, \emph{The geometry of syzygies}, ser. Graduate Texts in
  Mathematics.\hskip 1em plus 0.5em minus 0.4em\relax New York:
  Springer-Verlag, 2005, vol. 229, a second course in commutative algebra and
  algebraic geometry.

\bibitem{EH}
V.~Ene and J.~Herzog, \emph{{G}r\"obner Bases in Commutative Algebra}.\hskip
  1em plus 0.5em minus 0.4em\relax American Mathematical Soc., 2011, vol. 130.

\bibitem{EneHerzogHibi11:CM_binomial_edge_ideals}
V.~Ene, J.~Herzog, and T.~Hibi, ``Cohen-{M}acaulay binomial edge ideals,''
  \emph{Nagoya Mathematical Journal}, vol. 204, pp. 57--68, 2011.

\bibitem{EHHM}
V.~Ene, J.~Herzog, T.~Hibi, and F.~Mohammadi,
  ``\BIBforeignlanguage{English}{{Determinantal facet ideals}},''
  \emph{\BIBforeignlanguage{English}{{Mich. Math. J.}}}, vol.~62, no.~1, pp.
  39--57, 2013.

\bibitem{Ene}
V.~Ene, J.~Herzog, T.~Hibi, and A.~A. Qureshi, ``The binomial edge ideal of a
  pair of graphs,'' \emph{Nagoya Math. J.}, vol. 213, pp. 105--125, 2014.

\bibitem{mixed}
N.~Gonciulea and C.~Miller, ``Mixed ladder determinantal varieties,'' \emph{J.
  Algebra}, vol. 231, no.~1, pp. 104--137, 2000.

\bibitem{M2}
D.~R. Grayson and M.~E. Stillman, ``Macaulay2, a software system for research
  in algebraic geometry,'' Available at
  \url{http://www.math.uiuc.edu/Macaulay2/}.

\bibitem{HHBook}
\BIBentryALTinterwordspacing
J.~Herzog and T.~Hibi, \emph{Monomial ideals}, ser. Graduate Texts in
  Mathematics.\hskip 1em plus 0.5em minus 0.4em\relax London: Springer-Verlag
  London Ltd., 2011, vol. 260. [Online]. Available:
  \url{http://dx.doi.org/10.1007/978-0-85729-106-6}
\BIBentrySTDinterwordspacing

\bibitem{HH}
------, ``Ideals generated by adjacent 2-minors,'' \emph{J. Commut. Algebra},
  vol.~4, no.~4, pp. 525--549, 2012.

\bibitem{HHH}
J.~Herzog, T.~Hibi, F.~Hreinsd{\'o}ttir, T.~Kahle, and J.~Rauh, ``Binomial edge
  ideals and conditional independence statements,'' \emph{Adv. Appl. Math.},
  vol.~45, no.~3, pp. 317--333, 2010.

\bibitem{HT}
J.~Herzog and N.~Trung, ``{G}r{\"o}bner bases and multiplicity of determinantal
  and pfaffian ideals,'' \emph{Adv. Math.}, vol.~96, no.~1, pp. 1--37, 1992.

\bibitem{HSS}
S.~Ho{\c{s}}ten and S.~Sullivant, ``Ideals of adjacent minors,'' \emph{J.
  Algebra}, vol. 277, no.~2, pp. 615--642, 2004.

\bibitem{MR}
F.~Mohammadi and J.~Rauh, ``On determinantal hyperedge ideals,''
  \emph{Manuscript in preparation}, 2015.

\bibitem{FatemehLeila}
\BIBentryALTinterwordspacing
F.~Mohammadi and L.~Sharifan, ``Hilbert function of binomial edge ideals,''
  \emph{Comm. Algebra}, vol.~42, no.~2, pp. 688--703, 2014. [Online].
  Available: \url{http://dx.doi.org/10.1080/00927872.2012.721037}
\BIBentrySTDinterwordspacing

\bibitem{FatemehFarbod}
F.~Mohammadi and F.~Shokrieh, ``Divisors on graphs, connected flags, and
  syzygies,'' \emph{Int. Math. Res. Not. IMRN}, no.~24, pp. 6839--6905, 2014.

\bibitem{Rauh13}
J.~Rauh, ``Generalized binomial edge ideals,'' \emph{Adv. Appl. Math.}, vol.~3,
  no.~50, pp. 409--414, 2013.

\bibitem{St90}
B.~Sturmfels, ``{G}r{\"o}bner bases and {S}tanley decompositions of
  determinantal rings,'' \emph{Math. Z.}, vol. 205, no.~1, pp. 137--144, 1990.

\bibitem{SturmfelsZelevinsky}
B.~Sturmfels and A.~Zelevinsky, ``Maximal minors and their leading terms,''
  \emph{Adv. Math.}, vol.~98, no.~1, pp. 65--112, 1993.

\end{thebibliography}
\bibliographystyle{IEEEtranS}

\end{document}